\documentclass{birkjour}
\usepackage{amsmath}
\usepackage{amssymb}
\usepackage{amsthm}
\usepackage{mathrsfs}
\usepackage{hyperref}
\numberwithin{equation}{section}
\newtheorem{theorem}{Theorem}[section]
\newtheorem{lemma}[theorem]{Lemma}
\newtheorem{proposition}[theorem]{Proposition}

\newtheorem{definition}[theorem]{Definition}

\newtheorem{remark}[theorem]{Remark}

\newtheorem{main}[theorem]{Main Result}
\def\H{H_{\sigma}}

\def\L{L_{\sigma}}
\def\u{\mathbf{u}}
\def\b{\mathbf{b}}
\begin{document}
\title[Ideal Magnetic B\'{e}nard Problem]{On The Two and Three Dimensional Ideal Magnetic B\'{e}nard Problem  - Local Existence and Blow-up Criterion}

\author[Utpal Manna]{Utpal Manna}
%\author[iiser]{Utpal Manna\corref{cor1}}
%\ead{manna.utpal@iisertvm.ac.in}

%\author[iiser]{Akash A. Panda}
%\ead{akash.panda13@iisertvm.ac.in}


\address[iiser]{School of Mathematics, Indian Institute of Science Education and Research Thiruvananthapuram,
 695016, Kerala, INDIA}

\email{manna.utpal@iisertvm.ac.in}

\author[Akash A. Panda]{Akash A. Panda}

\address{%
School of Mathematics\\
Indian Institute of Science Education and Research (IISER) Thiruvananthapuram\\
Thiruvananthapuram 695016\\
Kerala, INDIA}

\email{akash.panda13@iisertvm.ac.in}



\begin{abstract}
In this paper, we consider the ideal magnetic B\'{e}nard problem in both two and three dimensions and prove local-in-time existence and uniqueness of strong solutions in $H^s$ for $s > \frac{n}{2}+1, n = 2,3$.  In addition, a necessary condition is derived for singularity development with respect to the $BMO$-norm of the vorticity and electrical current, generalising the Beale-Kato-Majda condition for ideal hydrodynamics. 
\end{abstract}

%\begin{keyword}
%{Magnetic B\'{e}nard problem \sep Commutator estimates \sep Blow-up criterion \sep Logarithmic Sobolev inequality}.
%\MSC[2010] 76D05 \sep 76D03
%\end{keyword}

\maketitle\setcounter{equation}{0}

\maketitle
\allowdisplaybreaks

\section{Introduction}

The magnetic B\'{e}nard problem with full viscosity is given by: 
\begin{align}
 \frac{\partial \mathbf u}{\partial t} + (\mathbf u \cdot \nabla)\mathbf u - \nu \Delta \mathbf u 
 + \nabla p_{\ast} &= (\mathbf b \cdot \nabla)\mathbf b + \mathbf \theta e_n,
 \ \ \textrm{in}\ \mathbb{R}^n\times (0, \infty),\label{MBF1}\\
\frac{\partial \mathbf \theta}{\partial t} + (\mathbf u \cdot \nabla)\mathbf \theta 
- \kappa \Delta \mathbf \theta &= \mathbf{u} \cdot e_n,  \ \ \textrm{in}\ \mathbb{R}^n\times (0, \infty),\label{MBF2}\\
 \frac{\partial \mathbf b}{\partial t} + (\mathbf u \cdot \nabla)\mathbf b 
 - \mu \Delta \mathbf b &= (\mathbf b \cdot \nabla)\mathbf u,  
 \ \ \textrm{in}\ \mathbb{R}^n\times (0, \infty),\label{MBF3}\\
\nabla \cdot \mathbf u = 0 &= \nabla \cdot \mathbf b,  \ \ \textrm{in}\ \mathbb{R}^n\times (0, \infty),
\end{align}
with initial conditions
\begin{align*}
\mathbf u(x,0) = \mathbf u_{0}(x) , \mathbf \theta(x,0) = \mathbf \theta_{0}(x ), 
\mathbf b(x,0) = \mathbf b_{0}(x)  \ \ \textrm{in}\ \mathbb{R}^n,
\end{align*}
where $n$ = 2, 3. Here
$\u:\mathbb{R}^n\times[0,\infty)\to\mathbb{R}^n$ is the velocity
field, $\mathbf{\theta}:\mathbb{R}^n\times[0,\infty)\to\mathbb{R}$ is the temperature, $\b:\mathbb{R}^n\times[0,\infty)\to\mathbb{R}^n$ is the
magnetic field, $p_*$ is the total pressure field
where $p_*=p+\frac{1}{2}|\b|^2$,
$p:\mathbb{R}^n\times[0,\infty)\to\mathbb{R}$ is the pressure. $e_n$ denotes the unit vector along the $n^{\textrm{th}}$ direction. The term $\mathbf \theta e_n$ represents buoyancy force on fluid motion and $\mathbf{u} \cdot e_n$ signifies the Rayleigh-B\'{e}nard convection in a heated inviscid fluid.
$\nu \geq 0$, $\mu \geq 0$ and $\kappa \geq 0$ denote the coefficients of kinematic viscosity, magnetic diffusion and thermal diffusion respectively.

The global-in-time regularity in two-dimensions of the above problem when $\nu, \mu$ and $\kappa > 0$ is known for a long time \cite{GP}. Due to the parabolic couplings, it is indeed possible to rewrite the above system in the abstract framework of the Navier-Stokes equations and then use the standard solvability techniques (e.g. see Temam \cite{Te}). In three-dimensions, one can at-most expect local-in-time solvability result with arbitrary initial data and global-in-time result for sufficiently small initial data, much like the Navier-Stokes equations. In \cite{CD}, the authors obtained the global well-posedness of two-dimensional magnetic B\'enard problem without thermal diffusivity and with vertical or horizontal magnetic diffusion. Moreover, the authors prove global regularity and some conditional regularity of strong solutions with mixed partial viscosity. This work provides an extension of an earlier result \cite{ZFN} on the global regularity with full dissipation and magnetic diffusion. It is worthwhile to note that there are very few literatures available where the case $\nu=\kappa=\mu=0$ has been discussed in two and three dimensions for the magnetic B\'{e}nard problem.

However, for the ideal magneto hydrodynamic (MHD) equations, i.e. when
$\mathbf\theta\equiv 0$ and $\nu=\mu=0$, in (\ref{MBF1})-(\ref{MBF3}), the local-in-time existence of
strong solutions have been proved by Schmidt \cite{SPG} and Secchi
\cite{SeP}, when the initial data is in $H^m$ for integer $m>1+n/2$.
Schmidt \cite{SPG} obtained the well-posedness and regularity of maximal solutions
and continuous dependence on forcing terms and initial data (using a
regularisation procedure). Caflisch, Klapper and Steele \cite{CKS}
derived a criteria for energy conservation and helicity conservation
for weak solutions of ideal MHD equations. The authors in \cite{CKS} extended
the Beale-Kato-Majda \cite{BKM} criterion to the three-dimensional ideal MHD
equations by  showing that for sufficiently regular initial data 
the following condition
$$\int_0^T\left(\|\nabla\times\u(\tau)\|_{\mathrm{L}^{\infty}}+\|\nabla\times\b(\tau)\|_{\mathrm{L}^{\infty}}\right) d\tau<\infty,$$ ensures that the solution can be continued beyond time $T$,
where $\nabla\times\u$ is the fluid vorticity, $\nabla\times\b$ is
the electrical current.

On the other hand, for the ideal Boussinesq system, i.e. when $\b\equiv \mathbf 0, \nu=\kappa=0$, and the Rayleigh-B\'{e}nard convection term $\u\cdot e_n$ is absent in (\ref{MBF1})-(\ref{MBF3}), only local-in-time existence results are available even in two-dimensions. It was proved in \cite{CN} that if the initial data $(\mathbf{u_0}, \mathbf{\theta}_0)\in\H^{3}(\mathbb R^{2})\times H^{3}(\mathbb R^{2})$, then local-in-time classical solutions exist and is unique. Moreover,  Beale-Kato-Majda type criterion for blow-up of smooth solutions is established in \cite{CN}. More precisely, they proved that the smooth solution exists on $[0, T]$ if and only if $\nabla\mathbf\theta\in L^1 (0, T; L^{\infty}(\mathbb{R}^2))$. For the three-dimensional Boussinesq system, a very few results on local-in-time existence and blow-up criterion are available (e.g. see \cite{GF, IM, QDY, Ye}). However, in the very particular case of the axisymmetric initial data, global-in-time well-posedness has been proven in three-dimensions by  Abidi et. al. \cite{AHK}. In a very recent work \cite{MP}, authors proved  local-in-time existence and uniqueness of strong solutions in $H^s$ for real $s > n/2+1$ for the ideal Boussinesq equations in $\mathbb{R}^n , n = 2,3$ and established Beale-Kato-Majda type blow-up criterion  with respect to the $BMO$-norm of the vorticity.

In this work, we consider the ideal magnetic B\'{e}nard problem (i.e. when $\nu=\kappa=\mu=0$) in both two and three dimensions and prove local-in-time existence and uniqueness of the strong solutions when the initial data $(\mathbf{u_0}, \mathbf{\theta}_0, \b_0)\in\H^{s}(\mathbb R^{n})\times H^{s}(\mathbb R^{n})\times\H^{s}(\mathbb R^{n})$, where $s>n/2+1$. We prove when $s>n/2+1$, $BMO$-norms of the vorticity, electrical current  and that of the gradient of the temperature (i.e. $\nabla\times\mathbf u, \nabla\times\mathbf b, \nabla\mathbf\theta \in L^1(0, T; BMO)$) control the breakdown of smooth solutions of the above systems.  However, we later show that under suitable additional assumption on $\mathbf{\theta}_0$, one can completely relax the condition on gradient of the temperature and the conditions $\nabla\times\mathbf u, \nabla\times\mathbf b\in L^1(0, T; BMO)$ are sufficient to ensure that the smooth solution persists. To the best of authors' knowledge, this work is new in the literature and may be seen as an extension of the blow-up criterion for ideal MHD equations due to Caflisch et. al. \cite{CKS} and that of ideal Boussinesq equations due to Manna et. al.  \cite{MP}.
 
We note that, in view of the recent work of Bourgain and Li \cite{BL} on the ill-posedness of the two and three dimensional Euler equations in $H^{n/2+1}, n=2, 3$, it seems likely that the ideal magnetic B\'{e}nard problem is also ill-posed in $H^{n/2+1}, n=2, 3$, although it still remains an open problem.

To be precise, in this work, we consider the following ideal magnetic B\'{e}nard problem
\begin{align}\label{mb1}
 \frac{\partial \mathbf u}{\partial t} + (\mathbf u \cdot \nabla)\mathbf u 
 + \nabla p_{\ast} &= (\mathbf b \cdot \nabla)\mathbf b + \mathbf \theta e_n,
 \ \ \textrm{in}\ \mathbb{R}^n\times (0, \infty),\\
\frac{\partial \mathbf \theta}{\partial t} + (\mathbf u \cdot \nabla)\mathbf \theta 
&= \mathbf{u} \cdot e_n,  \ \ \textrm{in}\ \mathbb{R}^n\times (0, \infty), 
 \label{mb2}\\
 \frac{\partial \mathbf b}{\partial t} + (\mathbf u \cdot \nabla)\mathbf b 
 &= (\mathbf b \cdot \nabla)\mathbf u,  \ \ \textrm{in}\ \mathbb{R}^n\times (0, \infty),\label{mb3}
\end{align}
with
\begin{align}
\nabla \cdot \mathbf u = 0 = \nabla \cdot \mathbf b,  \ \ \textrm{in}\ \mathbb{R}^n\times (0, \infty),\label{mb4}
\end{align}
\begin{align}
\mathbf u(x,0) = \mathbf u_{0}(x) , \mathbf \theta(x,0) = \mathbf \theta_{0}(x ), 
\mathbf b(x,0) = \mathbf b_{0}(x)  \ \ \textrm{in}\ \mathbb{R}^n, \label{mb5}
\end{align}
and prove the following main results.

First we state the result concerning local-in-time existence of strong solutions.
\begin{main}\label{mr1}
Let $s\in\mathbb{R}$ be such that $s > \frac{n}{2} + 1, n = 2, 3$. Let $(\mathbf{u_0}, \mathbf{\theta}_0,
\mathbf{b_0}) \in \H^{s}(\mathbb R^{n})\times H^{s}(\mathbb R^{n}) \times \H^{s}
(\mathbb R^{n})$. Then there exists a unique strong solution $(\mathbf{u}, \mathbf{\theta}, \mathbf{b})$ to the problem \eqref{mb1}-\eqref{mb5}, with $\mathbf{u}\in C([0, {T^{\ast}}]; \H^{s}(\mathbb R^{n}))$,

 \noindent $\mathbf{\theta} \in C([0, {T^{\ast}}]; H^{s}(\mathbb R^{n}))$ and $\mathbf{b}\in C([0, {T^{\ast}}]; \H^{s}
(\mathbb R^{n}))$ for some finite time $T^{\ast}=T^{\ast}(s, \|\mathbf u_0\|_{\H^{s}}, \|\mathbf \theta_0\|_{H^{s}},$ $\|\mathbf b_0\|_{\H^{s}}) > 0$.
\end{main}
To prove this result, we consider the Fourier truncated  ideal magnetic B\'{e}nard problem on the whole of $\mathbb{R}^n, n=2, 3$, and show that the solutions $(\mathbf u^{\textit R}, \mathbf \theta^{\textit R}, \mathbf b^{\textit R})$ of some smoothed version of  the  ideal magnetic B\'{e}nard system exist.  We then establish that the $H^s$-norm of $(\mathbf u^{\textit R}, \mathbf \theta^{\textit R},  \mathbf b^{\textit R})$ are uniformly bounded up to a terminal time $\tilde T$, which is independent of $R$.  We further show that up to the blowup time, the solution $(\mathbf u^{\textit R}, \mathbf \theta^{\textit R},  \mathbf b^{\textit R})$ is a Cauchy sequence in the $L^2$-norm as $R \to \infty$, and by using Sobolev interpolation, $(\mathbf u^{\textit R}, \mathbf \theta^{\textit R},  \mathbf b^{\textit R}) \to (\mathbf{u}, \mathbf{\theta}, \b)$ in any $H^{s'}$ for $0 < s' < s$. Finally we provide the proof of the above Main Result \ref{mr1} in Theorem \ref{mt1}.

Next, we establish that the $BMO$ norms of the vorticity and electrical current control the breakdown of smooth solutions. Our main result concerning the blow-up criterion is as follows:
\begin{main}\label{mr2}
Let $(\mathbf{u_0}, \mathbf{\theta}_0,
\mathbf{b_0})$ have same regularity as above and $s> \frac{n}{2} + 1, n = 2, 3$. If $(\mathbf u, \mathbf\theta, \mathbf b)$ satisfy the condition 
\begin{equation*}
\int_{0}^{T^{\ast}} \left( \|\nabla \times \mathbf u(\tau)\|_{BMO} + \|\nabla\mathbf\theta(\tau)\|_{BMO} + \|\nabla \times \mathbf b(\tau)\|_{BMO}\right) \, d\tau < \infty,
\end{equation*}
then the solution $(\mathbf u, \mathbf \theta, \mathbf b)$ can be continuously extended to $[0, T]$ for some $T > T^{\ast}.$ 
However, if $\mathbf{\theta}_0 \in H^{s}(\mathbb R^{n})\cap W^{1,p}(\mathbb R^{n}), \ 2\leq p \leq\infty$, then the condition
\[\int_{0}^{T^{\ast}} \left( \|\nabla \times \mathbf u(\tau)\|_{BMO} + \|\nabla \times \mathbf b(\tau)\|_{BMO}\right) \,d\tau 
< \infty\]
is sufficient to ensure that the solution $(\mathbf u, \mathbf \theta, \mathbf b)$  can be extended continuously to $[0, T]$ for some $T > T^{\ast}.$
\end{main}

\begin{remark}
The above result still holds if we replace $BMO$ with the Besov space $B^0_{\infty, \infty}$ used in Kozono et. al. \cite{KOT} or if we replace the condition by the one introduced in Planchon \cite {Pa}. 
To be precise, the condition above can be weakened to 
\begin{align*}
&\int_{0}^{T^{\ast}} \left(\|\nabla \times \mathbf u(\tau)\|_{B^0_{\infty, \infty}} + \|\nabla \times \mathbf b(\tau)\|_{B^0_{\infty, \infty}}\right) \,d\tau \\
&\quad = \int_{0}^{T^{\ast}} \left(\sup_{j} \|\triangle_j (\nabla \times \mathbf u(\tau))\|_{L^{\infty}} +\sup_{j} \|\triangle_j(\nabla \times \mathbf b(\tau))\|_{L^{\infty}}\right) \,d\tau
< \infty,
\end{align*}
or to 
\begin{align*}
\lim_{\delta\to 0} \int_{T^{\ast}-\delta}^{T^{\ast}} \left(\sup_{j} \|\triangle_j (\nabla \times \mathbf u(\tau))\|_{L^{\infty}} +\sup_{j} \|\triangle_j(\nabla \times \mathbf b(\tau))\|_{L^{\infty}}\right) \,d\tau
< \epsilon,
\end{align*}
for some sufficiently small $\epsilon >0$.
\end{remark}

\noindent 
The rest of the paper is organised as follows. After defining various operators, function spaces, its properties and certain basic inequalities in Section \ref{Pm}, we start investigating about the ideal magnetic B\'enard problem in Section \ref{EM} and prove results concerning energy estimates and convergence of the approximate solutions before proving the Main Result \ref{mr1} in Theorem \ref{mt1}. In section \ref{BC}, we provide the proof of the Main Result \ref{mr2} in Theorem \ref{buc} and Theorem \ref{buc1}.

\section{Preliminaries}\label{Pm}
\subsection{Fractional Derivative Operator.}

Let us define $J^s$ (real $s>0$), which denotes the Bessel potential of order $s,$ in terms of Fourier transform as follows:
\[ {\mathcal F} \left[ J^{s} f \right](\xi) = (1 + |\xi|^2)^{s/2} \widehat f(\xi) .\]

$J^s$ is also equivalent to the operator $(I - \Delta)^{s/2}.$ 

Assume $0<s<\infty$ and $f\in L^2(\mathbb{R}^n)$. Then $f\in H^{s}(\mathbb R^{n})$ if $(1 + |\xi|^2)^{s/2} \widehat f(\xi)\in L^2(\mathbb{R}^n)$. The norm on $H^{s}(\mathbb R^{n})$ is given by
\begin{align}\label{hs}
\| f\|_{H^{s}} = \left( \int_{\mathbb R^{n}} \left[ (1 + |\xi|^2)^{s/2} |\widehat f(\xi)| \right]^2 \right)^{1/2} = \left\| (1 + |\xi|^2)^{s/2} \widehat f(\xi)\right\|_{L^2}  = \| J^s f\|_{L^2}
\end{align}
and the inner product on $H^{s}(\mathbb R^{n})$ is given by
\begin{align*}
(f, g)_{H^{s}} = \left( (1 + |\xi|^2)^{s/2} \widehat f(\xi), (1 + |\xi|^2)^{s/2} \widehat g(\xi) \right)_{L^2} &= ({\mathcal F} \left[ J^s f \right](\xi), {\mathcal F} \left[ J^s g \right](\xi))_{L^2} \nonumber \\ &= \left( J^s f, J^s g \right)_{L^2}.
\end{align*}

\begin{remark}\label{gradhs}
It is trivial to observe that
 \[\| \nabla f\|_{{H}^{s-1}} \leq \|f\|_{{H}^{s}}.\]
\end{remark}

\subsection{Fourier Truncation Operator.}

Let us define the Fourier truncation operator $\mathcal S_{\textit{R}}$ as follows:
\[ \widehat{{\mathcal S_{\textit{R}}} f}(\xi) := \mathbf 1_{B_{R}}(\xi)\widehat f(\xi),\]

where $B_R$, a ball of radius $R$ centered at the origin and  $\mathbf 1_{B_{R}}$ is the indicator function. Then we infer the following important properties:

\begin{enumerate}
\item 
$\|{\mathcal S_{\textit{R}}} f \|_{H^{s}(\mathbb R^{n})} \leq C \|f \|_{H^{s}(\mathbb R^{n})}.$ (where $C$ is a constant independent of $R$)

\item
$\| {\mathcal S_{\textit{R}}} f  - f\|_{H^{s}(\mathbb R^{n})} \leq  \frac{C}{R^{k}} \|f \|_{H^{s+k}(\mathbb R^{n})}.$

\item
$\| ({\mathcal S_{\textit{R}}} - {\mathcal S_{\textit R'}}) f\|_{H^{s}} \leq C \left. \max\left\{ \left( \frac{1}{R} \right)^k , \left( \frac{1}{R'} \right)^k \right\} \right. \|f \|_{H^{s+k}}.$
\end{enumerate}
\noindent For the proofs of the properties see \cite{Fe}.

We define the function spaces
\[H^{s}_{\sigma}(\mathbb R^{n}) = \left\{ f \in H^{s}(\mathbb R^{n}) : \nabla \cdot f = 0 \right\},\left.\right.and \left.\right.\H^{s}(\mathbb R^{n}) = \left(H^{s}_{\sigma}(\mathbb R^{n})\right)^{n}.\]

\begin{remark}\label{prodhs}
If $s > n/2$, then each $f \in H^{s}(\mathbb R^{n})$ is bounded and continuous and hence 
\[ \|f\|_{L^{\infty}(\mathbb R^{n})} \leq C \|f\|_{H^s(\mathbb R^{n})}, \left. \right. for\left. \right. s > n/2.\]
Also, note that $H^s$ is an algebra for $s > n/2$, i.e., if $f, g \in H^{s}(\mathbb R^{n})$, then $fg \in H^{s}(\mathbb R^{n})$, for $s > n/2$. Hence, we have
\[ \| fg\|_{H^s} \leq C \| f\|_{H^s} \| g\|_{H^s}, \left. \right. for\left. \right.s > n/2.\]
\end{remark}

\begin{lemma}\label{div}
Fix $s > n/2$ and let $f \in H^{s}_{\sigma}$ and $g \in H^s$. Then
\[ \| (f \cdot \nabla) g \|_{H^{s-1}} \leq C \|f\|_{H^s} \|g\|_{H^s}.\]
\end{lemma}

\begin{proof}
Being in $H^{s}_{\sigma},$ $f$ is divergence free, and hence $(f \cdot \nabla) g = \nabla \cdot(f \otimes g)$. Rest of the proof is straightforward, since $H^s$ is an algebra for $s > n/2$.
\end{proof}

\begin{lemma}\label{Si}
$($Sobolev Inequality$)$ For $f \in H^{s}(\mathbb R^{n})$, we have
\[ \|f\|_{L^{q}(\mathbb R^{n})} \leq C_{n, s, q} \|f\|_{H^s(\mathbb R^{n})}\]
provided that q lies in the following range
\begin{enumerate}
\item[(i)]
if $s < n/2$, then $2 \leq q \leq \frac{2n}{n - 2s}$.
\item[(ii)]
if $s = n/2$, then $2 \leq q < \infty$.
\item[(iii)]
if $s > n/2$, then $2 \leq q \leq \infty$.
\end{enumerate}
\end{lemma}

\noindent For details see Kesavan \cite{Ks}.

\begin{remark}\label{r23}
We deduce the following result using Lemma \ref{Si}. For $n = 2$, we use  H\"{o}lder's inequality with exponents $2/\epsilon$ and $2/(1-\epsilon$), and Sobolev inequality for $0 <\epsilon< s-1$ to obtain
\[ \|fg\|_{L^2} \leq \|f\|_{L^{2/ {\epsilon}}}\|g\|_{L^{2/1-{\epsilon}}} \leq C \|f\|_{{\dot{H}}^{1-{\epsilon}}} \|g\|_{{\dot{H}}^{\epsilon}}  \leq C \|f\|_{H^1} \|g\|_{{H}^{s-1}}.\]

For $n = 3$, we use H\"{o}lder's inequality with exponents $6$ and $3$, and Sobolev inequality to obtain
\[ \|fg\|_{L^2} \leq \|f\|_{L^6} \|g\|_{L^3} \leq C \|f\|_{{\dot{H}}^{1}} \|g\|_{{\dot{H}}^{1/2}} \leq C \|f\|_{H^1} \|g\|_{H^{1/2}} \leq C \|f\|_{H^1} \|g\|_{{H}^{s-1}}.\]
We note that for both 2D and 3D we have the same estimate.
\end{remark}

\begin{lemma}\label{iss}
$($Interpolation in Sobolev spaces$)$. Given $s > 0,$ there exists a constant C depending on s, so that for all $f \in H^{s}(\mathbb R^{n})$ and $0 < s' < s,$
\[\|f\|_{H^{s'}} \leq C \|f\|_{L^2}^{1-s'/s}\|f\|^{s'/s}_{H^{s}}.\]
\end{lemma}

\noindent For details see \cite{AF} and for proof see Theorem 9.6, Remark 9.1 of \cite{LM}. 

\begin{lemma}\label{gni}
(Gagliardo-Nirenberg interpolation inequality  \cite{Ni}) Let $g \in L^{q}(\mathbb {R}^n)$ and its derivatives of order m, $D^{m} g \in L^{r}(\mathbb {R}^n)$, $1 \leq q, r \leq \infty$. For the derivatives $D^{j} g,$ $0 \leq j < m,$ the following inequality holds,

\[\|D^{j} g\|_{L^p} \leq C \|D^{m} g\|_{L^r}^{a} \|g\|_{L^q}^{1-a},\]
where
\[\frac{1}{p} = \frac{j}{n} + \left( \frac{1}{r} - \frac{m}{n} \right)a + \frac{1-a}{q},\]
for all $a$ in the interval
\[\frac{j}{m}\leq a < 1.\]
The constant $C$ depends only on $n, m, j, q, r, a.$
\end{lemma}

\subsection{Commutator Estimates.}

Let $f$ and $g$ are Schwartz class functions. Then for $s \geq 0$ we define,
\[[J^s, f]g = J^s(fg) - f(J^{s}g),\]
and
\begin{align}\label{cef}
[J^s, f] \nabla g = J^s((f \cdot \nabla)g) - (f \cdot \nabla)J^{s}g.
\end{align}

where $[J^s, f] = J^{s}f - f J^{s}$ is the commutator, in which $f$ is regarded as a multiplication operator.

\begin{lemma}\label{kpe}
 For $s \geq 0,$ and $1 < p < \infty$, we have a basic estimate
\[\|[J^s, f]g\|_{L^p} \leq C\left( \|\nabla f\|_{L^{\infty}}\|J^{s-1}g\|_{L^p} + \|J^{s}f\|_{L^p}\|g\|_{L^{\infty}}\right),\]

where C is a constant depending only on $n, p, s.$
\end{lemma}
\noindent For proof see the appendix of \cite{KP}.

\subsection{$BMO$ space and Logarithmic Sobolev inequality.}                                                                      
\begin{definition}
\textbf{The space $BMO$}(Bounded Mean Oscillation) is the Banach space of all functions $f \in L^{1}_{loc}(\mathbb R^{n})$ for which
\[\|f\|_{BMO} = \sup_{Q} \left( \frac{1}{|Q|} \int_{Q} \left| f(x) - f_{Q}\right| \,dx \right) < \infty,\]
where the sup ranges over all cubes $Q \subset \mathbb R^{n},$ and $ f_{Q}$ is the mean of f over $Q$.
\end{definition}
\noindent For more details see \cite{Fn}.

The space $BMO$ has two distinct advantageous properties compared to $L^{\infty}$. The first being the Riesz transforms are bounded in $BMO$ and secondly the singular integral operators of the Calderon-Zygmund type are also bounded in $BMO$. Hence, one can show that $\|\nabla \mathbf u\|_{BMO} \leq C \|\nabla \times \mathbf u\|_{BMO}$ (see \cite{KT}).

It is well known that the Sobolev space $W^{s, p}$ is embedded continuously into $L^{\infty}$ for $sp>n$.  However this embedding is false in the space $W^{k, r}$ when $kr=n$. Brezis-Gallouet \cite{BG} and Brezis-Wainger \cite{BW} provided the following inequality which relates the function spaces $L^{\infty}$ and $W^{s, p}$ at the critical value and was used to prove the existence of global solutions to the nonlinear Schr\"{o}dinger equations.
\begin{lemma}
Let $sp >n$. Then 
\[\|f\|_{L^{\infty}} \leq C \left( 1 + \log^{\frac{r-1}{r}} (1+\|f\|_{W^{s, p}})\right),\]
provided $\|f\|_{W^{k, r}} \leq 1$ for $kr=n$.
\end{lemma}
Similar embedding was investigated by Beale-Kato-Majda \cite{BKM} for vector functions to obtain the blow-up criterion of the solutions to the Euler equations.
\begin{lemma}\label{lsi}
Let $s > \frac{n}{p} + 1,$ then we have
\[\|\nabla f\|_{L^{\infty}} \leq C \left( 1 + \|\nabla \cdot f\|_{L^{\infty}} + \|\nabla \times f\|_{L^{\infty}}\left( 1 + log (e + \|f\|_{W^{s, p}})\right)\right),\]
 for all $f \in W^{s, p}(\mathbb R^{n}).$
\end{lemma}
Kozono and Taniuchi improved the above logarithmic Sobolev inequality in $BMO$ space, and applied the result to the three-dimensional Euler equations to prove that $BMO$-norm of the vorticity controls breakdown of smooth solutions.  
\begin{lemma}\label{kte}
Let $1 < p < \infty$ and let $s > \frac{n}{p}$, then we have
\[\|f\|_{L^{\infty}} \leq C \left(1 +\| f\|_{BMO}(1 + \log^{+} \|f\|_{W^{s,p}})\right),\]
 for all $f \in W^{s, p},$
 where $\log^{+} a = \log a$ if $a\geq 1$ and zero otherwise.
\end{lemma}
\noindent For proof see Theorem 1 of \cite{KT}.
\begin{remark} 
Throughout the following sections, $C$ denotes a generic constant.
\end{remark}

\section{Energy estimates, Local Existence and Uniqueness of magnetic B\'enard problem.}\label{EM}

We consider the following truncated ideal magnetic B\'enard problem on the whole of $\mathbb R^n$, for $n$ = $2, 3$:
\begin{align}\label{tr1}
\frac{\partial \mathbf u^{\textit R}}{\partial t} + {\mathcal S}_{\textit R} \left[ (\mathbf u^{\textit R} \cdot \nabla)\mathbf u^{\textit R} \right] + \nabla p^{\textit R} = \theta^{\textit R} e_n + {\mathcal S}_{\textit R} \left[ (\mathbf b^{\textit R} \cdot \nabla)\mathbf b^{\textit R} \right] ,
\end{align}
\begin{align}\label{tr2}
\frac{\partial \mathbf \theta^{\textit R}}{\partial t} + {\mathcal S}_{\textit R} \left[ (\mathbf u^{\textit R} \cdot \nabla)\mathbf \theta^{\textit R} \right] = \mathbf u^{\textit R} \cdot e_n  ,
\end{align}
\begin{align}\label{tr3}
\frac{\partial \mathbf b^{\textit R}}{\partial t} + {\mathcal S}_{\textit R} \left[ (\mathbf u^{\textit R} \cdot \nabla)\mathbf b^{\textit R} \right] = {\mathcal S}_{\textit R} \left[ (\mathbf b^{\textit R} \cdot \nabla)\mathbf u^{\textit R} \right]
 \end{align}
\begin{align}
\nabla \cdot \mathbf u^{\textit R} = 0 = \nabla \cdot \mathbf b^{\textit R},
\end{align}
\begin{align}\label{tr4}
\mathbf u^{\textit R}(0) = {\mathcal S}_{\textit R} \mathbf u_{0}, \mathbf \theta^{\textit R}(0) = {\mathcal S}_{\textit R} \mathbf \theta_{0}, \mathbf b^{\textit R}(0) = {\mathcal S}_{\textit R} \mathbf b_{0}.
\end{align}

As the truncations are invariant under the flow of the equation, by taking the truncated initial data we ensure that $\mathbf u^{\textit R}$, $\mathbf b^{\textit R}$ lie in the space 
\begin{align*}
V_R^{\sigma} := \left\{ g \in L^2(\mathbb R^n) : supp(\widehat g) \subset B_R, \nabla \cdot g = 0 \right\}
\end{align*}
and $\mathbf \theta^{\textit R}$ lie in the space
\[  V_R := \left\{ g \in L^2(\mathbb R^n) : supp(\widehat g) \subset B_R \right\} .\]  

The divergence free condition for $\mathbf u^{\textit R}$ can be obtained easily as
\[ \widehat {\nabla \cdot \mathbf u^{\textit R}}(\xi) = i\xi \cdot \mathbf 1_{B_R}(\xi) \widehat {\mathbf u} (\xi) = \mathbf 1_{B_R}(\xi) i\xi \cdot  \widehat {\mathbf u} (\xi) = \mathbf 1_{B_R}(\xi) \widehat {\nabla \cdot \mathbf u} (\xi)  = 0.  \]

Similarly we obtain divergence free condition for $\mathbf b^{\textit R}.$

\begin{proposition}\label{cut}
Let $(\mathbf u^{\textit R}$, $\mathbf b^{\textit R})$ $\in$ $\H^s(\mathbb R^n) \times \H^s(\mathbb R^n)$, for $s > n/2 +1$. Then the nonlinear operator $F(\mathbf u^{\textit R}, \mathbf b^{\textit R}) := {\mathcal S}_{\textit R} \left[ (\mathbf u^{\textit R} \cdot \nabla)\mathbf b^{\textit R} \right]$ is locally Lipschitz in $\mathbf u^{\textit R}$ and $\mathbf b^{\textit R}$ on the space $V_R^{\sigma}$.
\end{proposition}

\begin{proof} Let  $\mathbf b^{\textit R}$ $\in$ $\H^s(\mathbb R^n)$, for $s > n/2 + 1$. Then for proving $F(\cdot, \cdot)$ to be locally Lipschitz in $\mathbf u^{\textit R}$, we use integration by parts, H\"{o}lder's inequality and Lemma \ref{Si} to get,
\begin{align*}
&\left| \left(F(\mathbf u_1^{\textit R}, \mathbf b^{\textit R}) - F(\mathbf u_2^{\textit R}, \mathbf b^{\textit R}), \mathbf u_1^{\textit R} - \mathbf u_2^{\textit R} \right)_{L^2} \right| 
\\ & \quad \quad  \quad \quad \quad=  \left| \left({\mathcal S}_{\textit R} \left[ (\mathbf u_1^{\textit R} - \mathbf u_2^{\textit R})  \cdot \nabla \mathbf b^{\textit R} \right], \mathbf u_1^{\textit R} - \mathbf u_2^{\textit R} \right)_{L^2} \right|
\\ &\quad \quad  \quad \quad \quad= \left| \left( \left( (\mathbf u_1^{\textit R} - \mathbf u_2^{\textit R})  \cdot \nabla  \right)\mathbf b^{\textit R}, {\mathcal S}_{\textit R}( \mathbf u_1^{\textit R} - \mathbf u_2^{\textit R}) \right)_{L^2} \right|
\\ &\quad \quad  \quad \quad \quad= \left| -\left( \left( (\mathbf u_1^{\textit R} - \mathbf u_2^{\textit R}) \cdot \nabla \right)( \mathbf u_1^{\textit R} - \mathbf u_2^{\textit R}), {\mathcal S}_{\textit R}\mathbf b^{\textit R} \right)_{L^2} \right|
\\ &\quad \quad  \quad \quad \quad\leq \|\mathbf u_1^{\textit R} - \mathbf u_2^{\textit R}\|_{\L^2} \| \nabla(\mathbf u_1^{\textit R} - \mathbf u_2^{\textit R})\|_{\L^2} \|{\mathcal S}_{\textit R}\mathbf b^{\textit R} \|_{L^{\infty}}
\\ &\quad \quad  \quad \quad \quad\leq \|\mathbf u_1^{\textit R} - \mathbf u_2^{\textit R}\|_{\L^2} \| \mathbf u_1^{\textit R} - \mathbf u_2^{\textit R} \|_{\H^1} \|\mathbf b^{\textit R} \|_{L^{\infty}}
\\ &\quad \quad  \quad \quad \quad\leq C \|\mathbf u_1^{\textit R} - \mathbf u_2^{\textit R}\|_{\H^s} \|\mathbf b^{\textit R} \|_{\H^s} \|\mathbf u_1^{\textit R} - \mathbf u_2^{\textit R}\|_{\L^2}.
\end{align*}

This gives for $\mathbf b^{\textit R}$ $\in$ $\H^s(\mathbb R^n)$,
\begin{align*}
\| F(\mathbf u_1^{\textit R}, \mathbf b^{\textit R}) - F(\mathbf u_2^{\textit R}, \mathbf b^{\textit R})\|_{L^2} \leq C \| \mathbf b^{\textit R} \|_{\H^s} \|\mathbf u_1^{\textit R} - \mathbf u_2^{\textit R}\|_{\H^s}
\end{align*}
And hence $F(\cdot, \cdot)$ is locally Lipschitz in $\mathbf u^{\textit R}$. To prove $F$ to be locally Lipschitz in $\mathbf b^{\textit R}$, we use Remark \ref{div}. For $s > n/2 +1$ and $\mathbf u^{\textit R} \in \H^s(\mathbb R^{n}),$ we have
\begin{align*}
&\left| \left(F(\mathbf u^{\textit R}, \mathbf b_1^{\textit R}) - F(\mathbf u^{\textit R}, \mathbf b_2^{\textit R}), \mathbf b_1^{\textit R} - \mathbf b_2^{\textit R} \right)_{L^2} \right| 
\\ &\quad \quad  \quad \quad \quad= \left| \left({\mathcal S}_{\textit R} (\mathbf u^{\textit R} \cdot \nabla) (\mathbf b_1^{\textit R}-\mathbf b_2^{\textit R}), \mathbf b_1^{\textit R} - \mathbf b_2^{\textit R} \right)_{L^2} \right|
\\ &\quad \quad  \quad \quad \quad= \left| \left( (\mathbf u^{\textit R} \cdot \nabla) (\mathbf b_1^{\textit R}-\mathbf b_2^{\textit R}), {\mathcal S}_{\textit R}(\mathbf b_1^{\textit R} - \mathbf b_2^{\textit R}) \right)_{L^2} \right|
\\ &\quad \quad  \quad \quad \quad\leq \| (\mathbf u^{\textit R} \cdot \nabla) (\mathbf b_1^{\textit R}-\mathbf b_2^{\textit R}) \|_{\L^2} \|{\mathcal S}_{\textit R}(\mathbf b_1^{\textit R} - \mathbf b_2^{\textit R}) \|_{\L^2}
\\ &\quad \quad  \quad \quad \quad\leq C \| \mathbf u^{\textit R}\|_{\H^1} \| \nabla (\mathbf b_1^{\textit R}-\mathbf b_2^{\textit R}) \|_{\H^{s-1}} \| \mathbf b_1^{\textit R} - \mathbf b_2^{\textit R} \|_{\L^2}
\\ &\quad \quad  \quad \quad \quad\leq C \| \mathbf u^{\textit R}\|_{\H^s} \| \mathbf b_1^{\textit R}-\mathbf b_2^{\textit R} \|_{\H^{s}} \| \mathbf b_1^{\textit R} - \mathbf b_2^{\textit R} \|_{\L^2}
\end{align*}

Hence for $\mathbf u^{\textit R}$ $\in$ $\H^s(\mathbb R^n)$, we have
\begin{align*}
\| \left(F(\mathbf u^{\textit R}, \mathbf b_1^{\textit R}) - F(\mathbf u^{\textit R}, \mathbf b_2^{\textit R}\right) \|_{L^2} \leq C \| \mathbf u^{\textit R}\|_{\H^s} \| \mathbf b_1^{\textit R}-\mathbf b_2^{\textit R} \|_{\H^{s}}
\end{align*}
And hence $F(\cdot, \cdot)$ is locally Lipschitz in $\mathbf b^{\textit R}$.
\end{proof}

Similarly one can show $F(\mathbf b^{\textit R}, \mathbf u^{\textit R})$ is locally Lipschitz in $\mathbf b^{\textit R}$ and $\mathbf u^{\textit R}$ on the space $V_R^{\sigma}\times V_R^{\sigma}$ and $F(\mathbf u^{\textit R}, \mathbf \theta^{\textit R})$ is locally Lipschitz in $\mathbf u^{\textit R}$ and $\mathbf \theta^{\textit R}$ on the space $V_R^{\sigma}\times V_R$. 

Hence by Picard's theorem for infinite dimensional ordinary differential equations, there exist a solution $(\mathbf u^{\textit R}, \mathbf \theta^{\textit R}, \mathbf b^{\textit R})$ in $V_R^{\sigma} \times V^R \times V_R^{\sigma}$ for some interval $[0, T],$ where $T$ depends on $R$. Moreover, the solution will exist as long as $\| \mathbf u^{\textit R}\|_{\H^s}$ , $ \| \mathbf \theta^{\textit R}\|_{H^s}$ and $\| \mathbf b^{\textit R}\|_{\H^s}$ remain finite.

\subsection{Energy Estimates.}\label{ee}

In this section we will obtain $L^2$ and $H^s, s>n/2+1,$ energy estimates for $\mathbf u^{\textit R}, \mathbf \theta^{\textit R}$ and $\mathbf b^{\textit R}$. In the course of proving the $\| \mathbf u^{\textit R}\|_{\H^s}$ , $ \| \mathbf \theta^{\textit R}\|_{H^s}$  and $\| \mathbf b^{\textit R}\|_{\H^s}$ are uniformly bounded, we will pick up a blow-up time $T^{\ast}.$

\begin{proposition}\label{lfin}
$($$L^2$-Energy Estimate$)$ Given $(\mathbf u_{0}, \mathbf \theta_{0}, \mathbf b_{0})$ $\in$ $\L^2(\mathbb R^n) \times L^2(\mathbb R^n) \times \L^2(\mathbb R^n)$ with $s > n/2 + 1$, then for any $t \in [0, T]$, where $0 < T < \infty,$ we have
\[\sup_{t\in [0, T]} \left(\| \mathbf u^{\textit R}(t)\|^2_{\L^2}+\| \mathbf \theta^{\textit R}(t)\|^2_{L^2}+\| \mathbf b^{\textit R}(t)\|^2_{\L^2}\right) < C\]
where C depends only on $\|\mathbf u_0\|_{\L^2}, \| \mathbf \theta_0\|_{L^2}, \| \mathbf b_0\|_{\L^2}$ and $T$.
\end{proposition}

\begin{proof}
Consider the equations \eqref{tr1}-\eqref{tr3}. Taking $L^2-$ inner product of \eqref{tr1}, \eqref{tr2} and \eqref{tr3} with $\mathbf u^{\textit R}, \theta^{\textit R}$ and $\mathbf b^{\textit R}$ respectively, and adding we obtain
\begin{align}\label{ddtut}
\frac{1}{2} \frac{d}{dt} \left( \|\mathbf u^{\textit R}\|^2_{\L^2} + \|\mathbf \theta^{\textit R}\|^2_{L^2} + \|\mathbf b^{\textit R}\|^2_{\L^2}\right) = \left(\theta^{\textit R} e_n, \mathbf u^{\textit R}\right)_{L^2} + \left((\mathbf u^{\textit R} \cdot e_n), \theta^{\textit R}\right)_{L^2}.
\end{align}
In the above calculation, we have used the fact that $\left( (\mathbf u^{\textit R} \cdot \nabla)\mathbf u^{\textit R}, \mathbf u^{\textit R}\right)_{L^2},$ 

\noindent $\left( (\mathbf u^{\textit R} \cdot \nabla)\mathbf \theta^{\textit R}, \mathbf \theta^{\textit R} \right)_{L^2}$ and $\left((\mathbf u^{\textit R} \cdot \nabla)\mathbf b^{\textit R}, \mathbf b^{\textit R}\right)_{L^2}$ vanish and $\left((\mathbf b^{\textit R} \cdot \nabla)\mathbf b^{\textit R}, \mathbf u^{\textit R}\right)_{L^2} = - \left((\mathbf b^{\textit R} \cdot \nabla)\mathbf u^{\textit R}, \mathbf b^{\textit R}\right)_{L^2}$.
It is easy to see that,
\begin{align*}
|\left(\theta^{\textit R} e_n, \mathbf u^{\textit R}\right)_{L^2}| \leq \|\theta^{\textit R} e_n\|_{L^2} \|\mathbf u^{\textit R}\|_{L^2} &\leq \|\theta^{\textit R} \|_{L^2} \|\mathbf u^{\textit R}\|_{L^2} 
\\ &\leq \frac{1}{2} \left( \|\mathbf u^{\textit R}\|^2_{L^2} + \|\theta^{\textit R} \|^2_{L^2} + \|\mathbf b^{\textit R}\|^2_{L^2}\right),
\end{align*}
and
\begin{align*}
|\left((\mathbf u^{\textit R} \cdot e_n), \theta^{\textit R}\right)_{L^2}| \leq \|\mathbf u^{\textit R}\|_{L^2}  \|\theta^{\textit R} \|_{L^2} \leq \frac{1}{2} \left( \|\mathbf u^{\textit R}\|^2_{L^2} + \|\theta^{\textit R} \|^2_{L^2} + \|\mathbf b^{\textit R}\|^2_{L^2}\right).
\end{align*}
Using the above estimates in \eqref{ddtut} and letting $Y(t) = \|\mathbf u^{\textit R}(t)\|^2_{\L^2} + \|\mathbf \theta^{\textit R}(t)\|^2_{L^2} + \|\mathbf b^{\textit R}(t)\|^2_{\L^2}$, we obtain
\begin{align*}
\frac{dY(t)}{dt} \leq 2Y(t).
\end{align*}
Straightforward integration and the fact that $\|\u^{\textit R}(0)\|_{\L^2}  \leq \|u_0\|_{\L^2}$,

\noindent $\|\theta^{\textit R}(0)\|_{L^2}  \leq \|\theta_0\|_{L^2}$ and $\|\b^{\textit R}(0)\|_{\L^2}  \leq \|\b_0\|_{\L^2}$ yield

\[\sup_{t\in [0, T]}Y(t) \leq C(\|\mathbf u_0\|_{\L^2}, \| \mathbf \theta_0\|_{L^2}, \| \mathbf b_0\|_{\L^2}, T)\]
So we have the desired result.
\end{proof}

\begin{proposition}\label{fin}
Let $(\mathbf u_{0}, \mathbf \theta_{0}, \mathbf b_{0})$ $\in$ $\H^s(\mathbb R^n) \times H^s(\mathbb R^n) \times \H^s(\mathbb R^n)$ with $s > n/2 + 1$. Then there exists a time $T^{\ast} = T^{\ast}(s, \|u_0\|_{\H^s}, \|\theta_0\|_{H^s}, \|\b_0\|_{\H^s}) > 0$ such that the following norms
\[ \sup_{t \in [0, T^{\ast}]} \| \mathbf u^{\textit R}(t)\|_{\H^s}, \quad  \sup_{t \in [0, T^{\ast}]}\| \mathbf \theta^{\textit R}(t)\|_{H^s}, \quad \sup_{t \in [0, T^{\ast}]} \| \mathbf b^{\textit R}(t)\|_{\H^s}\]
are bounded uniformly in $R$.
\end{proposition}

\begin{proof} 
Let $J^s$ denote the fractional derivative operator as defined earlier.

Now for $s > n/2 +1$, apply $J^s$ to all the equations \eqref{tr1} - \eqref{tr3}:

\begin{align}\label{jtr1}
\frac{\partial (J^{s} \mathbf u^{\textit R})}{\partial t} + {\mathcal S}_{\textit R} J^s \left[ (\mathbf u^{\textit R} \cdot \nabla)\mathbf u^{\textit R} \right] + \nabla J^{s} p^{\textit R} = J^{s}(\theta^{\textit R} e_n) + {\mathcal S}_{\textit R} J^{s} \left[ (\mathbf b^{\textit R} \cdot \nabla)\mathbf b^{\textit R} \right] ,
\end{align}
\begin{align}\label{jtr2}
\frac{\partial (J^{s} \mathbf \theta^{\textit R})}{\partial t} + {\mathcal S}_{\textit R} J^{s} \left[ (\mathbf u^{\textit R} \cdot \nabla)\mathbf \theta^{\textit R} \right] = J^{s} (\mathbf u^{\textit R} \cdot e_n)  ,
\end{align}
\begin{align}\label{jtr3}
\frac{\partial (J^{s} \mathbf b^{\textit R})}{\partial t} + {\mathcal S}_{\textit R} J^{s} \left[ (\mathbf u^{\textit R} \cdot \nabla)\mathbf b^{\textit R} \right] = {\mathcal S}_{\textit R} J^{s} \left[ (\mathbf b^{\textit R} \cdot \nabla)\mathbf u^{\textit R} \right]
\end{align}

Taking ${L^2}$-inner product of \eqref{jtr1},  \eqref{jtr2} and \eqref{jtr3}  with $J^s \mathbf u^{\textit R}$,  $J^s \mathbf \theta^{\textit R}$ and $J^s \mathbf b^{\textit R}$ respectively, we obtain

\begin{align}\label{ijtr1}
\left(\frac{\partial (J^{s} \mathbf u^{\textit R})}{\partial t}, J^s \mathbf u^{\textit R}\right)_{L^2} &+ \left({\mathcal S}_{\textit R} J^s \left[ (\mathbf u^{\textit R} \cdot \nabla)\mathbf u^{\textit R} \right], J^s \mathbf u^{\textit R}\right)_{L^2} + \left(\nabla J^{s} p^{\textit R}, J^s \mathbf u^{\textit R}\right)_{L^2} \nonumber
\\ &= \left(J^{s}(\theta^{\textit R} e_n), J^s \mathbf u^{\textit R}\right)_{L^2} + \left({\mathcal S}_{\textit R} J^{s} \left[ (\mathbf b^{\textit R} \cdot \nabla)\mathbf b^{\textit R} \right], J^s \mathbf u^{\textit R}\right)_{L^2} ,
\end{align}
\begin{align}\label{ijtr2}
\left(\frac{\partial (J^{s} \mathbf \theta^{\textit R})}{\partial t}, J^s \mathbf \theta^{\textit R}\right)_{L^2} + \left({\mathcal S}_{\textit R} J^{s} \left[ (\mathbf u^{\textit R} \cdot \nabla)\mathbf \theta^{\textit R} \right], J^s \mathbf \theta^{\textit R}\right)_{L^2} = \left(J^{s} (\mathbf u^{\textit R} \cdot e_n), J^s \mathbf \theta^{\textit R}\right)_{L^2}  ,
\end{align}
\begin{align}\label{ijtr3}
\left(\frac{\partial (J^{s} \mathbf b^{\textit R})}{\partial t}, J^s \mathbf b^{\textit R}\right)_{L^2} &+ \left({\mathcal S}_{\textit R} J^{s} \left[ (\mathbf u^{\textit R} \cdot \nabla)\mathbf b^{\textit R} \right], J^s \mathbf b^{\textit R}\right)_{L^2} \nonumber
\\ &= \left({\mathcal S}_{\textit R} J^{s} \left[ (\mathbf b^{\textit R} \cdot \nabla)\mathbf u^{\textit R} \right], J^s \mathbf b^{\textit R}\right)_{L^2}.
\end{align}

\noindent We estimate each term of \eqref{ijtr1}, \eqref{ijtr2} and \eqref{ijtr3} separately,

\begin{enumerate}
\item
$\left(\frac{\partial (J^{s} \mathbf u^{\textit R})}{\partial t}, J^s \mathbf u^{\textit R}\right)_{L^2}$
\[= \int_{B_R} \frac{\partial J^s \mathbf u^{\textit R}}{\partial t} J^s \mathbf u^{\textit R} \,dx = \frac{1}{2} \int_{B_R} \frac{\partial \left| J^s \mathbf u^{\textit R} \right|^2}{\partial t} = \frac{1}{2} \frac{d}{dt} \|J^s \mathbf u^{\textit R}\|^2_{\L^2} =  \frac{1}{2} \frac{d}{dt} \|\mathbf u^{\textit R}\|^2_{\H^s}.\]
\item
Applying weak Parseval's identity and using the fact that ${\mathcal S}_{\textit R} \mathbf u^{\textit R} = \mathbf u^{\textit R}$, since $\mathbf u^{\textit R} \in V^{\sigma}_R$ we get,
\[ \left({\mathcal S}_{\textit R} J^s \left[ (\mathbf u^{\textit R} \cdot \nabla)\mathbf u^{\textit R} \right], J^s \mathbf u^{\textit R}\right)_{L^2} = \left( J^s \left[ (\mathbf u^{\textit R} \cdot \nabla)\mathbf u^{\textit R} \right], J^s \mathbf u^{\textit R}\right)_{L^2}. \]

Using definition of commutator and incompressibility of $\mathbf u^{\textit R}$, we obtain
\begin{align*}
 \left( \left[J^s, \mathbf u^{\textit R}\right] \nabla \mathbf u^{\textit R}, J^s \mathbf u^{\textit R}\right)_{L^2} &= \left( J^s \left[ (\mathbf u^{\textit R} \cdot \nabla)\mathbf u^{\textit R} \right] - \left(\mathbf u^{\textit R} \cdot \nabla\right) J^s\mathbf u^{\textit R}, J^s \mathbf u^{\textit R}\right)_{L^2}
\\ &= \left( J^s \left[ (\mathbf u^{\textit R} \cdot \nabla)\mathbf u^{\textit R} \right], J^s \mathbf u^{\textit R}\right)_{L^2}.
\end{align*}

Now using Lemma \ref{kpe} and H\"older's inequality we obtain,
\begin{align*}
&\left| \left( \left[J^s, \mathbf u^{\textit R}\right] \nabla \mathbf u^{\textit R}, J^s \mathbf u^{\textit R}\right)_{L^2}\right| \leq \|\left[J^s, \mathbf u^{\textit R}\right] \nabla \mathbf u^{\textit R}\|_{L^2} \| J^s \mathbf u^{\textit R}\|_{L^2}
\\ &\quad\quad\quad\quad\quad\quad\leq C \left( \| \nabla \mathbf u^{\textit R}\|_{L^{\infty}} \| J^{s-1}\nabla \mathbf u^{\textit R}\|_{\L^{2}} + \| J^{s}\mathbf u^{\textit R}\|_{\L^{2}} \|\nabla \mathbf u^{\textit R}\|_{L^{\infty}}\right) \| \mathbf u^{\textit R}\|_{\H^{s}}
\\ &\quad\quad\quad\quad\quad\quad\leq C \left( \| \nabla \mathbf u^{\textit R}\|_{\H^{s-1}} \| \nabla \mathbf u^{\textit R}\|_{\H^{s-1}} + \| \mathbf u^{\textit R}\|_{\H^{s}} \|\nabla \mathbf u^{\textit R}\|_{\H^{s-1}}\right) \| \mathbf u^{\textit R}\|_{\H^{s}}
\\ &\quad\quad\quad\quad\quad\quad\leq C \left( \| \mathbf u^{\textit R}\|^2_{\H^{s}} + \| \mathbf u^{\textit R}\|^2_{\H^{s}} \right) \| \mathbf u^{\textit R}\|_{\H^{s}}
\\ &\quad\quad\quad\quad\quad\quad\leq C \left(  \| \mathbf u^{\textit R}\|^2_{\H^{s}} + \| \mathbf \theta^{\textit R}\|^2_{H^{s}} + \| \mathbf b^{\textit R}\|^2_{\H^{s}}\right) \| \mathbf u^{\textit R}\|_{\H^{s}}.
 \end{align*}

\item
Using integration by parts and then divergence free condition on $\mathbf u^{\textit R}$ yields
\[\left(\nabla J^{s} p^{\textit R}, J^s \mathbf u^{\textit R}\right)_{L^2} = \left( J^{s} p^{\textit R}, J^s \nabla \cdot \mathbf u^{\textit R}\right)_{L^2} = 0.\]

\item
Using H\"older's inequality and then Young's inequality to the term
\begin{align*}
\left|\left(J^{s}(\theta^{\textit R} e_n), J^s \mathbf u^{\textit R}\right)_{L^2}\right| &\leq \|J^{s}(\theta^{\textit R} e_n)\|_{L^2} \|J^s \mathbf u^{\textit R}\|_{\L^2}
\\ &\leq \|\theta^{\textit R} e_n\|_{H^s} \| \mathbf u^{\textit R}\|_{\H^s} 
\\ &\leq C \left( \| \mathbf u^{\textit R}\|^{2}_{\H^s} + \| \mathbf \theta^{\textit R}\|^{2}_{H^s} + \| \mathbf b^{\textit R}\|^{2}_{\H^s}\right).
 \end{align*}

\item Using property of the bilinear operator, we have
\begin{align}
&\left({\mathcal S}_{\textit R} J^{s} \left[ (\mathbf b^{\textit R} \cdot \nabla)\mathbf b^{\textit R} \right], J^s \mathbf u^{\textit R}\right)_{L^2} = \left( J^{s} \left[ (\mathbf b^{\textit R} \cdot \nabla)\mathbf b^{\textit R} \right], J^s {\mathcal S}_{\textit R} \mathbf u^{\textit R}\right)_{L^2} \nonumber
\\ &= \left( J^{s} \left[ (\mathbf b^{\textit R} \cdot \nabla)\mathbf b^{\textit R} \right], J^s \mathbf u^{\textit R}\right)_{L^2} = - \left( J^{s} \left[ (\mathbf b^{\textit R} \cdot \nabla)\mathbf u^{\textit R} \right], J^s \mathbf b^{\textit R}\right)_{L^2}.\nonumber
 \end{align}
 
\item 
Calculation similar to (1) gives
\[\left(\frac{\partial (J^{s} \mathbf \theta^{\textit R})}{\partial t}, J^s \mathbf \theta^{\textit R}\right)_{L^2} = \frac{1}{2} \frac{d}{dt} \|\mathbf \theta^{\textit R}\|^2_{H^s}.\]

\item
Following similar calculation as in (2) and using Lemma \ref{kpe} to the term

 $\left({\mathcal S}_{\textit R} J^s \left[ (\mathbf u^{\textit R} \cdot \nabla)\mathbf \theta^{\textit R} \right], J^s \mathbf \theta^{\textit R}\right)_{L^2}$ we get,
\begin{align*}
&\left| \left( \left[J^s, \mathbf u^{\textit R}\right] \nabla \mathbf \theta^{\textit R}, J^s \mathbf \theta^{\textit R}\right)_{L^2}\right| \leq \|\left[J^s, \mathbf u^{\textit R}\right] \nabla \mathbf \theta^{\textit R}\|_{L^2} \|J^s \mathbf \theta^{\textit R}\|_{L^2}
\\ &\quad\quad\quad\quad\quad\quad\leq C \left( \| \nabla \mathbf u^{\textit R}\|_{L^{\infty}} \| J^{s-1}\nabla \mathbf \theta^{\textit R}\|_{L^{2}} + \| J^{s}\mathbf u^{\textit R}\|_{\L^{2}} \|\nabla \mathbf \theta^{\textit R}\|_{L^{\infty}}\right) \| \mathbf \theta^{\textit R}\|_{H^{s}}
\\ &\quad\quad\quad\quad\quad\quad\leq C \left( \| \nabla \mathbf u^{\textit R}\|_{\H^{s-1}} \| \nabla \mathbf \theta^{\textit R}\|_{H^{s-1}} + \| \mathbf u^{\textit R}\|_{\H^{s}} \|\nabla \mathbf \theta^{\textit R}\|_{H^{s-1}}\right) \| \mathbf \theta^{\textit R}\|_{H^{s}}
\\ &\quad\quad\quad\quad\quad\quad\leq C \left(  \| \mathbf u^{\textit R}\|^2_{\H^{s}} + \| \mathbf \theta^{\textit R}\|^2_{H^{s}} + \| \mathbf b^{\textit R}\|^2_{\H^{s}}\right) \| \mathbf \theta^{\textit R}\|_{H^{s}}.
 \end{align*}

\item
 Application of H\"older's inequality and Young's inequality yield
 \begin{align*}
\left| \left(J^{s} (\mathbf u^{\textit R} \cdot e_n), J^{s} \mathbf \theta^{\textit R}\right)_{L^2} \right| &\leq \|J^{s} (\mathbf u^{\textit R} \cdot e_n)\|_{\L^2} \| J^{s} \mathbf \theta^{\textit R}\|_{L^2} \leq \|\mathbf u^{\textit R}\|_{\H^s} \| \mathbf \theta^{\textit R}\|_{H^s}
\\ &\leq C \left( \| \mathbf u^{\textit R}\|^{2}_{\H^s} + \| \mathbf \theta^{\textit R}\|^{2}_{H^s} + \| \mathbf b^{\textit R}\|^{2}_{\H^s}\right).
 \end{align*}

\item
Similarly, $\left(\frac{\partial (J^{s} \mathbf b^{\textit R})}{\partial t}, J^s \mathbf b^{\textit R}\right)_{L^2} = \frac{1}{2} \frac{d}{dt} \|\mathbf b^{\textit R}\|^2_{\H^s}.$

\item
Following similar steps as in (7), replacing $\mathbf \theta^{\textit R}$ by $\mathbf b^{\textit R}$ we obtain
\[\left|\left({\mathcal S}_{\textit R} J^{s} \left[ (\mathbf u^{\textit R} \cdot \nabla)\mathbf b^{\textit R} \right], J^s \mathbf b^{\textit R}\right)_{L^2} \right|\leq  C \left(  \| \mathbf u^{\textit R}\|^2_{\H^{s}} + \| \mathbf \theta^{\textit R}\|^2_{H^{s}} + \| \mathbf b^{\textit R}\|^2_{\H^{s}}\right) \| \mathbf b^{\textit R}\|_{\H^{s}}.\]

\item
Weak Parseval's identity gives,
\[\left({\mathcal S}_{\textit R} J^{s} \left[ (\mathbf b^{\textit R} \cdot \nabla)\mathbf u^{\textit R} \right], J^s \mathbf b^{\textit R}\right)_{L^2} = \left( J^{s} \left[ (\mathbf b^{\textit R} \cdot \nabla)\mathbf u^{\textit R} \right], J^s \mathbf b^{\textit R}\right)_{L^2}\]
\end{enumerate}

Now adding  \eqref{ijtr1}, \eqref{ijtr2} and \eqref{ijtr3} (using the estimates obtained through (1) to (11)) we have,
 \begin{align*}
 &\frac{1}{2} \frac{d}{dt} \left( \| \mathbf u^{\textit R}\|^2_{\H^{s}} + \| \mathbf \theta^{\textit R}\|^2_{H^{s}} + \| \mathbf b^{\textit R}\|^2_{\H^{s}}\right) 
\\ &\leq C \left( \| \mathbf u^{\textit R}\|^2_{\H^{s}} + \| \mathbf \theta^{\textit R}\|^2_{H^{s}} + \| \mathbf b^{\textit R}\|^2_{\H^{s}}\right) \left(\| \mathbf u^{\textit R}\|_{\H^{s}} + \| \mathbf \theta^{\textit R}\|_{H^{s}} + \| \mathbf b^{\textit R}\|_{\H^{s}}\right)
\\&\leq\frac{C}{2}\left( \| \mathbf u^{\textit R}\|^2_{\H^{s}} + \| \mathbf \theta^{\textit R}\|^2_{H^{s}} + \| \mathbf b^{\textit R}\|^2_{\H^{s}}\right)^2 + \frac{3C}{2}\left( \| \mathbf u^{\textit R}\|^2_{\H^{s}} + \| \mathbf \theta^{\textit R}\|^2_{H^{s}} + \| \mathbf b^{\textit R}\|^2_{\H^{s}}\right).
\end{align*}

\noindent Now letting $X(t) =  \| \mathbf u^{\textit R}(t)\|^2_{\H^{s}} + \| \mathbf \theta^{\textit R}(t)\|^2_{H^{s}} + \| \mathbf b^{\textit R}(t)\|^2_{\H^{s}}$, we have,
 \begin{align*}
\frac{d}{dt} X(t) \leq 3C X(t) + X(t)^2 \leq \frac{3}{2}C^2 + (\frac{3}{2}+C) X(t)^2.
\end{align*}

So for all $0 \leq t \leq T,$
\begin{align*}
X(t) \leq X_0 + \frac{3}{2}C^2 + (\frac{3}{2}+C) \int_0^t  X(s)^2 \, ds.
\end{align*}

Now applying Bihari's inequality \cite{DD}, we have

\[X(t) \leq \frac{\frac{3}{2}C^2 + X_0}{1- (\frac{3}{2}C^2 + X_0)(\frac{3}{2}+C)T}.\]

Note that $\|\u^{\textit R}(0)\|_{\H^s}  \leq \|u_0\|_{\H^s}$, $\|\theta^{\textit R}(0)\|_{H^s}  \leq \|\theta_0\|_{H^s}$ and $\|\b^{\textit R}(0)\|_{\H^s}  \leq \|\b_0\|_{\H^s}$. So provided we choose $ T^{\ast} < \frac{1}{(\frac{3}{2}C^2 + X_0)(\frac{3}{2}+C)}$, $\| \mathbf u^{\textit R}\|_{\H^s}, \| \mathbf \theta^{\textit R} \|_{H^s}$  and $\| \mathbf b^{\textit R}\|_{\H^s}$ remain bounded on $[0, T^{\ast}]$ independent of $R$.

\end{proof}

\subsection{Local Existence and Uniqueness.}\label{EU}

In this subsection, we will prove existence and uniqueness of the local-in time strong solution of the magnetic B\'enard problem \eqref{mb1}-\eqref{mb4}.

\begin{proposition}\label{Ca}
The family $(\mathbf u^{\textit R}, \mathbf \theta^{\textit R}, \mathbf b^{\textit R})$ of solutions of the magnetic B\'enard problem \eqref{tr1}-\eqref{tr4} are Cauchy as $R \to \infty$ in the space $L^{\infty}\left([0, T^{\ast}]; \L^2(\mathbb R^{n})\right) \times$ 

\noindent $ L^{\infty}\left([0, T^{\ast}]; L^2(\mathbb R^{n})\right)\times L^{\infty}\left([0, T^{\ast}]; \L^2(\mathbb R^{n})\right)$ .
\end{proposition}
\begin{proof}
Consider the equations \eqref{tr1}, \eqref{tr2} and \eqref{tr3}. Then taking the difference between the equations for $R$ and $R'$ with $R' > R$ we get, 
\begin{align}\label{dif1}
&\frac{\partial }{\partial t}\left(\mathbf u^{\textit R} - \mathbf u^{\textit R'}\right) + \nabla \left(p^{\textit R} - p^{\textit R'}\right)\nonumber\\ 
&\quad = \theta^{\textit R} e_n - \theta^{\textit R'} e_n - {\mathcal S}_{\textit R} \left[ (\mathbf u^{\textit R} \cdot \nabla)\mathbf u^{\textit R} \right] + {\mathcal S}_{\textit R'} \left[ (\mathbf u^{\textit R'} \cdot \nabla)\mathbf u^{\textit R'} \right] \nonumber\\
&\quad\quad + {\mathcal S}_{\textit R} \left[ (\mathbf b^{\textit R} \cdot \nabla)\mathbf b^{\textit R} \right] - {\mathcal S}_{\textit R'} \left[ (\mathbf b^{\textit R'} \cdot \nabla)\mathbf b^{\textit R'} \right],
\end{align}
\begin{align}\label{dif2}
\frac{\partial}{\partial t} \left(\mathbf \theta^{\textit R}- \mathbf \theta^{\textit R'}\right) + {\mathcal S}_{\textit R} \left[ (\mathbf u^{\textit R} \cdot \nabla)\mathbf \theta^{\textit R} \right] - {\mathcal S}_{\textit R'} \left[ (\mathbf u^{\textit R'} \cdot \nabla)\mathbf \theta^{\textit R'} \right] = \mathbf u^{\textit R} \cdot e_n - \mathbf u^{\textit R'} \cdot e_n  ,
\end{align}
\begin{align}\label{dif3}
&\frac{\partial }{\partial t} \left(\mathbf b^{\textit R} - \mathbf b^{\textit R'}\right)+ {\mathcal S}_{\textit R}\left[ (\mathbf u^{\textit R} \cdot \nabla)\mathbf b^{\textit R} \right] -  {\mathcal S}_{\textit R'} \left[ (\mathbf u^{\textit R'} \cdot \nabla)\mathbf b^{\textit R'} \right]\nonumber\\
&\quad = {\mathcal S}_{\textit R} \left[ (\mathbf b^{\textit R} \cdot \nabla)\mathbf u^{\textit R} \right]-  {\mathcal S}_{\textit R'} \left[ (\mathbf b^{\textit R'} \cdot \nabla)\mathbf u^{\textit R'} \right].
 \end{align}
Taking inner product of \eqref{dif1}, \eqref{dif2} and \eqref{dif3} with $\mathbf u^{\textit R} - \mathbf u^{\textit R'}$, $\mathbf \theta^{\textit R}- \mathbf \theta^{\textit R'}$ and $\mathbf b^{\textit R} - \mathbf b^{\textit R'}$ respectively, and then adding we get

\begin{align}\label{utbse}
\frac{1}{2}&\frac{d}{dt} \left( \|\mathbf u^{\textit R} - \mathbf u^{\textit R'}\|^2_{\L^2} + \|\mathbf \theta^{\textit R} - \mathbf \theta^{\textit R'}\|^2_{L^2} + \|\mathbf b^{\textit R} - \mathbf b^{\textit R'}\|^2_{\L^2}\right) \nonumber
\\ &=\left(\theta^{\textit R} e_n - \theta^{\textit R'} e_n, \mathbf u^{\textit R} -  \mathbf u^{\textit R'}\right) - \left( \mathbf u^{\textit R} \cdot e_n - \mathbf u^{\textit R'} \cdot e_n , \mathbf \theta^{\textit R}- \mathbf \theta^{\textit R'}\right)\nonumber
\\ &\quad\quad-\underbrace{\left(  {\mathcal S}_{\textit R} \left[ (\mathbf u^{\textit R} \cdot \nabla)\mathbf u^{\textit R} \right] - {\mathcal S}_{\textit R'} \left[ (\mathbf u^{\textit R'} \cdot \nabla)\mathbf u^{\textit R'} \right], \mathbf u^{\textit R} - \mathbf u^{\textit R'}\right)}_{I_1}\nonumber
\\ &\quad\quad+ \underbrace{\left(  {\mathcal S}_{\textit R} \left[ (\mathbf b^{\textit R} \cdot \nabla)\mathbf b^{\textit R} \right] - {\mathcal S}_{\textit R'} \left[ (\mathbf b^{\textit R'} \cdot \nabla)\mathbf b^{\textit R'} \right], \mathbf u^{\textit R} - \mathbf u^{\textit R'}\right)}_{I_2}\nonumber
\\ &\quad\quad-\underbrace{\left(  {\mathcal S}_{\textit R} \left[ (\mathbf u^{\textit R} \cdot \nabla)\mathbf \theta^{\textit R} \right] - {\mathcal S}_{\textit R'} \left[ (\mathbf u^{\textit R'} \cdot \nabla)\mathbf \theta^{\textit R'} \right], \mathbf \theta^{\textit R} - \mathbf \theta^{\textit R'}\right)}_{I_3}\nonumber
\\ &\quad\quad+ \underbrace{\left(  {\mathcal S}_{\textit R} \left[ (\mathbf b^{\textit R} \cdot \nabla)\mathbf u^{\textit R} \right] - {\mathcal S}_{\textit R'} \left[ (\mathbf b^{\textit R'} \cdot \nabla)\mathbf u^{\textit R'} \right], \mathbf b^{\textit R} - \mathbf b^{\textit R'}\right)}_{I_4}\nonumber
\\ &\quad\quad-\underbrace{\left(  {\mathcal S}_{\textit R} \left[ (\mathbf u^{\textit R} \cdot \nabla)\mathbf b^{\textit R} \right] - {\mathcal S}_{\textit R'} \left[ (\mathbf u^{\textit R'} \cdot \nabla)\mathbf b^{\textit R'} \right], \mathbf b^{\textit R} - \mathbf b^{\textit R'}\right)}_{I_5}
 \end{align}
We will calculate each term on the right hand side of \eqref{utbse} separately. First observe that,
\begin{align}\label{I0}
\left|\left(\theta^{\textit R} e_n - \theta^{\textit R'} e_n, \mathbf u^{\textit R} -  \mathbf u^{\textit R'}\right)\right| &\leq \|\theta^{\textit R} e_n - \theta^{\textit R'} e_n\|_{L^2} \|\mathbf u^{\textit R} -  \mathbf u^{\textit R'}\|_{\L^2}\nonumber 
\\ &\leq \|\theta^{\textit R} - \theta^{\textit R'}\|_{L^2} \|\mathbf u^{\textit R} -  \mathbf u^{\textit R'}\|_{\L^2},
\end{align}
and
\begin{align}\label{I00}
\left| \left( \mathbf u^{\textit R} \cdot e_n - \mathbf u^{\textit R'} \cdot e_n , \mathbf \theta^{\textit R} - \mathbf \theta^{\textit R'}\right) \right| &\leq \|\mathbf u^{\textit R} \cdot e_n - \mathbf u^{\textit R'} \cdot e_n\|_{\L^2} \| \mathbf \theta^{\textit R} - \mathbf \theta^{\textit R'}\|_{L^2}\nonumber
\\ &\leq  \|\mathbf u^{\textit R} - \mathbf u^{\textit R'} \|_{\L^2} \| \mathbf \theta^{\textit R} - \mathbf \theta^{\textit R'}\|_{L^2}.
\end{align}
We split $I_1 = \left( {\mathcal S}_{\textit R} \left[ (\mathbf u^{\textit R} \cdot \nabla)\mathbf u^{\textit R} \right] - {\mathcal S}_{\textit R'} \left[ (\mathbf u^{\textit R'} \cdot \nabla)\mathbf u^{\textit R'} \right], \mathbf u^{\textit R} - \mathbf u^{\textit R'}\right)$ in to three parts:
\begin{align}\label{sur} 
&I_1=\left( ({\mathcal S}_{\textit R}-{\mathcal S}_{\textit R'}) \left[ (\mathbf u^{\textit R} \cdot \nabla)\mathbf u^{\textit R} \right], \mathbf u^{\textit R} - \mathbf u^{\textit R'}\right) \nonumber
\\ &\quad+ \left( {\mathcal S}_{\textit R'} \left[ ((\mathbf u^{\textit R}-\mathbf u^{\textit R'}) \cdot \nabla)\mathbf u^{\textit R} \right], \mathbf u^{\textit R} - \mathbf u^{\textit R'}\right)\nonumber
\\ &\quad + \left( {\mathcal S}_{\textit R'} \left[ (\mathbf u^{\textit R'} \cdot \nabla)(\mathbf u^{\textit R} - \mathbf u^{\textit R'}) \right], \mathbf u^{\textit R} - \mathbf u^{\textit R'}\right).
\end{align}

\noindent For $R' > R$, using the property of Fourier truncation operator provided $0 < \epsilon < s-1$, the first term of \eqref{sur} becomes
\begin{align}\label{sruuu}
&\left| \left( ({\mathcal S}_{\textit R}-{\mathcal S}_{\textit R'}) \left[ (\mathbf u^{\textit R} \cdot \nabla)\mathbf u^{\textit R} \right], \mathbf u^{\textit R} - \mathbf u^{\textit R'}\right) \right|\nonumber
\\&\quad\leq \| ({\mathcal S}_{\textit R}-{\mathcal S}_{\textit R'}) \left[ (\mathbf u^{\textit R} \cdot \nabla)\mathbf u^{\textit R} \right] \|_{\L^2} \| \mathbf u^{\textit R} - \mathbf u^{\textit R'}\|_{\L^2}  \nonumber
\\ &\quad \leq \frac{C}{R^\epsilon}\| (\mathbf u^{\textit R} \cdot \nabla)\mathbf u^{\textit R}\|_{\H^\epsilon} \| \mathbf u^{\textit R} - \mathbf u^{\textit R'}\|_{\L^2} = \frac{C}{R^\epsilon}\| \nabla \cdot (\mathbf u^{\textit R} \otimes \mathbf u^{\textit R})\|_{\H^\epsilon} \| \mathbf u^{\textit R} - \mathbf u^{\textit R'}\|_{\L^2} \nonumber
\\ &\quad \leq \frac{C}{R^\epsilon}\| \mathbf u^{\textit R} \otimes \mathbf u^{\textit R}\|_{\H^s} \| \mathbf u^{\textit R} - \mathbf u^{\textit R'}\|_{\L^2} \leq \frac{C}{R^\epsilon}\|\mathbf u^{\textit R}\|^2_{\H^s} \| \mathbf u^{\textit R} - \mathbf u^{\textit R'}\|_{\L^2}.
\end{align}

\noindent Now for $s > n/2 + 1,$ the second term of \eqref{sur}
\begin{align}\label{sruu}
&\left|\left( {\mathcal S}_{\textit R'} \left[ ((\mathbf u^{\textit R}-\mathbf u^{\textit R'}) \cdot \nabla)\mathbf u^{\textit R} \right], \mathbf u^{\textit R} - \mathbf u^{\textit R'}\right)\right| \nonumber
 \\ &\quad\leq  \|((\mathbf u^{\textit R}-\mathbf u^{\textit R'}) \cdot \nabla)\mathbf u^{\textit R}\|_{\L^2} \|\mathbf u^{\textit R}-\mathbf u^{\textit R'}\|_{\L^2} \nonumber
\\ &\quad\leq \|\mathbf u^{\textit R}-\mathbf u^{\textit R'}\|_{\L^2} \|\nabla \mathbf u^{\textit R}\|_{L^{\infty}} \|\mathbf u^{\textit R}-\mathbf u^{\textit R'}\|_{\L^2} \nonumber
 \\ &\quad\leq \|\nabla \mathbf u^{\textit R}\|_{\H^{s-1}} \|\mathbf u^{\textit R}-\mathbf u^{\textit R'}\|^2_{\L^2}\leq \|\mathbf u^{\textit R}\|_{\H^{s}} \|\mathbf u^{\textit R}-\mathbf u^{\textit R'}\|^2_{\L^2}.
\end{align}

\noindent Using weak Parseval's identity, integration by parts and divergence free condition on $\mathbf u^{\textit R}$ and $\mathbf u^{\textit R'}$ to the third term of \eqref{sur} we get,
\begin{align*}
\left( {\mathcal S}_{\textit R'} \left[ (\mathbf u^{\textit R'} \cdot \nabla)(\mathbf u^{\textit R} - \mathbf u^{\textit R'}) \right], \mathbf u^{\textit R} - \mathbf u^{\textit R'}\right) = 0.
\end{align*}
Therefore we obtain, using \eqref{sruuu}, \eqref{sruu} in \eqref{sur},
\begin{align}\label{I1}
|I_1| \leq \frac{C}{R^\epsilon}\|\mathbf u^{\textit R}\|^2_{\H^s} \| \mathbf u^{\textit R} - \mathbf u^{\textit R'}\|_{\L^2} + \|\mathbf u^{\textit R}\|_{\H^{s}} \|\mathbf u^{\textit R}-\mathbf u^{\textit R'}\|^2_{\L^2}.
\end{align}
Similarly we split $I_3$ and $I_5$ to obtain, 
\begin{align}\label{I3}
|I_3| \leq \frac{C}{\textit R^{\epsilon}} \| \mathbf u^{\textit R}\|_{\H^s} \| \mathbf \theta^{\textit R} \|_{H^s} \| \mathbf \theta^{\textit R} - \mathbf \theta^{\textit R'}\|_{L^2} +  \| \mathbf u^{\textit R}-\mathbf u^{\textit R'}\|_{\L^2}  \|\theta^{\textit R}\|_{H^s} \|\mathbf \theta^{\textit R} - \mathbf \theta^{\textit R'}\|_{L^2}.
\end{align}
and
\begin{align}\label{I5}
|I_5| \leq \frac{C}{\textit R^{\epsilon}} \| \mathbf u^{\textit R}\|_{\H^s} \| \mathbf b^{\textit R} \|_{\H^s} \| \mathbf b^{\textit R} - \mathbf b^{\textit R'}\|_{\L^2} +  \| \mathbf u^{\textit R}-\mathbf u^{\textit R'}\|_{\L^2}  \|\b^{\textit R}\|_{\H^s} \|\mathbf b^{\textit R} - \mathbf b^{\textit R'}\|_{\L^2}.
\end{align}
We spilt $I_2$ and $I_4$ in the similar manner. However, note that one term of $I_2$ will cancel with one term of $I_4$ due to 
\begin{align*}
\left( {\mathcal S}_{\textit R'} \left[ (\mathbf b^{\textit R'} \cdot \nabla)(\mathbf b^{\textit R} - \mathbf b^{\textit R'}) \right], \mathbf u^{\textit R} - \mathbf u^{\textit R'}\right) = -\left( (\mathbf b^{\textit R'} \cdot \nabla)(\mathbf u^{\textit R} - \mathbf u^{\textit R'}), \mathbf b^{\textit R} - \mathbf b^{\textit R'}\right).
\end{align*}
Therefore we have,
\begin{align}\label{I2}
I_2 \leq \frac{C}{\textit R^{\epsilon}}\|\mathbf b^{\textit R}\|^{2}_{\H^{s}}  \|\mathbf u^{\textit R} - \mathbf u^{\textit R'}\|_{\L^2} + \|\mathbf b^{\textit R}\|_{\H^{s}} \| \mathbf b^{\textit R}-\mathbf b^{\textit R'}\|_{\L^2}  \|\mathbf u^{\textit R} - \mathbf u^{\textit R'}\|_{\L^2},
\end{align}
and
\begin{align}\label{I4}
I_4\leq \frac{C}{\textit R^{\epsilon}}\|\mathbf b^{\textit R}\|_{\H^{s}}  \|\mathbf u^{\textit R}\|_{\H^{s}} \|\mathbf b^{\textit R} - \mathbf b^{\textit R'}\|_{\L^2} + \|\mathbf u^{\textit R}\|_{\H^{s}} \| \mathbf b^{\textit R}-\mathbf b^{\textit R'}\|^{2}_{\L^2}.
\end{align}
\noindent Using the estimates obtained in \eqref{I0}, \eqref{I00}, \eqref{I1}-\eqref{I4} in \eqref{utbse}, we have
\begin{align*}
\frac{1}{2}&\frac{d}{dt} \left( \|\mathbf u^{\textit R} - \mathbf u^{\textit R'}\|^2_{\L^2} + \|\mathbf \theta^{\textit R} - \mathbf \theta^{\textit R'}\|^2_{L^2} + \|\mathbf b^{\textit R} - \mathbf b^{\textit R'}\|^2_{\L^2}\right)
\\ &\leq \frac{C}{\textit R^{\epsilon}}  \|\mathbf u^{\textit R}\|^{2}_{\H^{s}} \|\mathbf u^{\textit R} - \mathbf u^{\textit R'}\|_{\L^2} + \|\mathbf u^{\textit R}\|_{\H^{s}} \|\mathbf u^{\textit R} - \mathbf u^{\textit R'}\|^{2}_{\L^2} 
\\ &\quad +  \frac{C}{\textit R^{\epsilon}}\|\mathbf u^{\textit R}\|_{\H^{s}}  \|\mathbf \theta^{\textit R}\|_{H^{s}} \|\mathbf \theta^{\textit R} - \mathbf \theta^{\textit R'}\|_{L^2}+ \|\mathbf b^{\textit R}\|_{\H^{s}} \| \mathbf u^{\textit R}-\mathbf u^{\textit R'}\|_{\L^2} \| \mathbf b^{\textit R}-\mathbf b^{\textit R'}\|_{\L^2}
\\ &\quad+ \|\mathbf \theta^{\textit R}\|_{H^{s}} \| \mathbf u^{\textit R}-\mathbf u^{\textit R'}\|_{\L^2} \| \mathbf \theta^{\textit R}-\mathbf \theta^{\textit R'}\|_{L^2} + \frac{C}{\textit R^{\epsilon}}\|\mathbf b^{\textit R}\|_{\H^{s}}  \|\mathbf u^{\textit R}\|_{\H^{s}} \|\mathbf b^{\textit R} - \mathbf b^{\textit R'}\|_{\L^2}
\\ &\quad+ \frac{C}{\textit R^{\epsilon}}\|\mathbf b^{\textit R}\|^{2}_{\H^{s}}  \|\mathbf u^{\textit R} - \mathbf u^{\textit R'}\|_{\L^2} + \|\mathbf b^{\textit R}\|_{\H^{s}} \| \mathbf u^{\textit R}-\mathbf u^{\textit R'}\|_{\L^2}  \| \mathbf b^{\textit R}-\mathbf b^{\textit R'}\|_{\L^2}
\\ &\quad+  \frac{C}{\textit R^{\epsilon}}\|\mathbf b^{\textit R}\|_{\H^{s}}  \|\mathbf u^{\textit R}\|_{\H^{s}} \|\mathbf b^{\textit R} - \mathbf b^{\textit R'}\|_{\L^2} + \|\mathbf u^{\textit R}\|_{\H^{s}} \| \mathbf b^{\textit R}-\mathbf b^{\textit R'}\|^{2}_{\L^2}
\\ &\quad+ 2 \|\mathbf \theta^{\textit R} - \mathbf \theta^{\textit R'}\|_{L^2} \|\mathbf u^{\textit R} - \mathbf u^{\textit R'}\|_{\L^2}.
\end{align*}
Applying Proposition \ref{fin} and Young's inequality and rearranging the terms we obtain,
\begin{align}\label{ddt}
\frac{d}{dt}& \left( \|\mathbf u^{\textit R} - \mathbf u^{\textit R'}\|^2_{\L^2} + \|\mathbf \theta^{\textit R} - \mathbf \theta^{\textit R'}\|^2_{L^2} + \|\mathbf b^{\textit R} - \mathbf b^{\textit R'}\|^2_{\L^2}\right)\nonumber
\\ &\leq \frac{C_1}{\textit R^{\epsilon}} \|\mathbf u^{\textit R} - \mathbf u^{\textit R'}\|_{\L^2} + C_2 \|\mathbf u^{\textit R} - \mathbf u^{\textit R'}\|^{2}_{\L^2} + 2 \left(\|\mathbf \theta^{\textit R} - \mathbf \theta^{\textit R'}\|^2_{L^2} + \|\mathbf u^{\textit R} - \mathbf u^{\textit R'}\|^2_{\L^2}\right) \nonumber
\\ &\quad\quad+  \frac{C_3}{\textit R^{\epsilon}} \|\mathbf \theta^{\textit R} - \mathbf \theta^{\textit R'}\|_{L^2} + C_4 \left(\|\mathbf \theta^{\textit R} - \mathbf \theta^{\textit R'}\|^{2}_{L^2}  + \|\mathbf u^{\textit R} - \mathbf u^{\textit R'}\|^2_{\L^2} \right)\nonumber
\\ &\quad\quad+ \frac{C_5}{\textit R^{\epsilon}} \|\mathbf b^{\textit R} - \mathbf b^{\textit R'}\|_{\L^2}
+ C_6 \left( \| \mathbf u^{\textit R}-\mathbf u^{\textit R'}\|^2_{\L^2} + \| \mathbf b^{\textit R}-\mathbf b^{\textit R'}\|^2_{\L^2} \right)\nonumber
\\ &\quad\quad+ \frac{C_7}{\textit R^{\epsilon}} \|\mathbf u^{\textit R} - \mathbf u^{\textit R'}\|_{\L^2} + C_8\left( \| \mathbf u^{\textit R}-\mathbf u^{\textit R'}\|^2_{\L^2} +  \| \mathbf b^{\textit R}-\mathbf b^{\textit R'}\|^2_{\L^2} \right) \nonumber
\\ &\quad\quad +C_{10} \| \mathbf b^{\textit R}-\mathbf b^{\textit R'}\|^{2}_{\L^2} \nonumber
\\ &\leq \frac{M}{\textit R^{\epsilon}} \left( \|\mathbf u^{\textit R} - \mathbf u^{\textit R'}\|_{\L^2} + \|\mathbf \theta^{\textit R} - \mathbf \theta^{\textit R'}\|_{L^2} + \|\mathbf b^{\textit R} - \mathbf b^{\textit R'}\|_{\L^2}\right)\nonumber
\\ &\quad\quad+ M \left( \|\mathbf u^{\textit R} - \mathbf u^{\textit R'}\|^2_{\L^2} + \|\mathbf \theta^{\textit R} - \mathbf \theta^{\textit R'}\|^2_{L^2} + \|\mathbf b^{\textit R} - \mathbf b^{\textit R'}\|^2_{\L^2}\right).
\end{align}
Let $Y(t) = \|\mathbf u^{\textit R} - \mathbf u^{\textit R'}\|_{\L^2} + \|\mathbf \theta^{\textit R} - \mathbf \theta^{\textit R'}\|_{L^2} + \|\mathbf b^{\textit R} - \mathbf b^{\textit R'}\|_{\L^2}$, then
\begin{align*}
\|\mathbf u^{\textit R} - \mathbf u^{\textit R'}\|^2_{\L^2} &+ \|\mathbf \theta^{\textit R} - \mathbf \theta^{\textit R'}\|^2_{L^2} + \|\mathbf b^{\textit R} - \mathbf b^{\textit R'}\|^2_{\L^2}
\\ &\leq Y(t)^2 \leq 3 \left( \|\mathbf u^{\textit R} - \mathbf u^{\textit R'}\|^2_{\L^2} + \|\mathbf \theta^{\textit R} - \mathbf \theta^{\textit R'}\|^2_{L^2} + \|\mathbf b^{\textit R} - \mathbf b^{\textit R'}\|^2_{\L^2}\right).
\end{align*}
So
\begin{align*}
\frac{d}{dt} \left( Y(t)^2 \right) \leq 3 \frac{d}{dt} \left( \|\mathbf u^{\textit R} - \mathbf u^{\textit R'}\|^2_{\L^2} + \|\mathbf \theta^{\textit R} - \mathbf \theta^{\textit R'}\|^2_{L^2} + \|\mathbf b^{\textit R} - \mathbf b^{\textit R'}\|^2_{\L^2}\right)
\end{align*}
\begin{align*}
2Y\frac{dY}{dt} \leq 3 \frac{d}{dt} \left( \|\mathbf u^{\textit R} - \mathbf u^{\textit R'}\|^2_{\L^2} + \|\mathbf \theta^{\textit R} - \mathbf \theta^{\textit R'}\|^2_{L^2} + \|\mathbf b^{\textit R} - \mathbf b^{\textit R'}\|^2_{\L^2}\right)
\end{align*}
Then from \eqref{ddt} we obtain,
\begin{align*}
\frac{dY}{dt} \leq MY + \frac{M}{\textit R^{\epsilon}}.
\end{align*}
Finally applying Gronwall's lemma, we observe
\begin{align}\label{ytr}
 \sup_{t \in [0, T^{\ast}]} Y(t) \leq \frac{C(M, T^{\ast})}{R^\epsilon} \to 0,
\end{align}
as $R\to\infty$ (as $R' > R$,  $R' \to \infty$ as well), concluding that $(\mathbf u^{\textit R}, \mathbf \theta^{\textit R}, \mathbf b^{\textit R})$ are Cauchy in $L^{\infty}\left([0, T^{\ast}]; \L^2(\mathbb R^n)\right) \times L^{\infty}\left([0, T^{\ast}]; L^2(\mathbb R^n)\right) \times L^{\infty}\left([0, T^{\ast}]; \L^2(\mathbb R^n)\right)$ as $R \to \infty$.
\end{proof}

\begin{proposition}
For any $s' > n/2 +1$ with $s' < s$, $(\mathbf u^{\textit R}, \mathbf \theta^{\textit R}, \mathbf b^{\textit R}) \to (\mathbf u, \mathbf \theta, \mathbf b)$ in 

\noindent $L^{\infty}\left([0, T^{\ast}]; \H^{s'}(\mathbb R^n)\right) \times L^{\infty}\left([0, T^{\ast}]; H^{s'}(\mathbb R^n)\right) \times L^{\infty}\left([0, T^{\ast}]; \H^{s'}(\mathbb R^n)\right).$ 
\end{proposition}

\begin{proof}
From Proposition \ref{Ca} we conclude that $(\mathbf u^{\textit R}, \mathbf \theta^{\textit R}, \mathbf b^{\textit R}) \to (\mathbf u, \mathbf \theta, \mathbf b)$ strongly in 

$L^{\infty}\left([0, T^{\ast}]; \L^2(\mathbb R^n)\right) \times L^{\infty}\left([0, T^{\ast}]; L^2(\mathbb R^n)\right) \times L^{\infty}\left([0, T^{\ast}]; \L^2(\mathbb R^n)\right).$

Using Lemma \ref{iss} for $s' < s$ and $s' > n/2 +1,$
\begin{align*}
\sup_{t \in [0, T^{\ast}]} \|\mathbf b^{\textit R} - \mathbf b\|_{\H^{s'}} &\leq C \sup_{t \in [0, T^{\ast}]} \left(\|\mathbf b^{\textit R} - \mathbf b\|_{\L^{2}}^{1-s'/s} \|\mathbf b^{\textit R} - \mathbf b\|_{\H^{s}}^{s'/s}\right)
\\ &\leq C \left( \sup_{t \in [0, T^{\ast}]} \|\mathbf b^{\textit R} - \mathbf b\|_{\L^{2}}\right)^{\frac{1-s'}{s}} \left( \sup_{t \in [0, T^{\ast}]}  \|\mathbf b^{\textit R} - \mathbf b\|_{\H^{s}}\right)^{\frac{s'}{s}}.
\end{align*}

From Proposition \ref{fin} and Proposition \ref{Ca} we obtain
\begin{align*}
\sup_{t \in [0,T^{\ast}]} \|\mathbf b^{\textit R} - \mathbf b\|_{\H^{s'}} \leq M  \left( \sup_{t \in [0,T^{\ast}]} \|\mathbf b^{\textit R} - \mathbf b\|_{\L^{2}}\right)^{1-s'/s} \to 0 \quad \quad as \quad R \to \infty.
\end{align*}

So we get,
\begin{align}\label{urcu}
\mathbf b^{\textit R} \to \mathbf b\left.\right.\left.\right. in \left.\right.\left.\right. L^{\infty}\left([0, T^{\ast}]; \H^{s'}(\mathbb R^n)\right),
\end{align}

Similarly we can show
\begin{align*}
 \mathbf \theta^{\textit R} \to  \mathbf \theta \quad in \quad L^{\infty}\left([0, T^{\ast}]; H^{s'}(\mathbb R^n)\right),\quad \mathbf u^{\textit R} \to \mathbf u \quad in\quad L^{\infty} \left([0, T^{\ast}]; \H^{s'}(\mathbb R^n)\right).
\end{align*}
\end{proof}

\begin{proposition}
For any $s' > n/2 +1$, as $R \to \infty$ the non-linear terms 
\begin{align*}
&{\mathcal S}_{\textit R} \left[ (\mathbf u^{\textit R} \cdot \nabla)\mathbf u^{\textit R} \right] \to (\mathbf u \cdot \nabla)\mathbf u,  {\mathcal S}_{\textit R} \left[ (\mathbf b^{\textit R} \cdot \nabla)\mathbf b^{\textit R} \right] \to (\mathbf b \cdot \nabla)\mathbf b,  {\mathcal S}_{\textit R} \left[ (\mathbf u^{\textit R} \cdot \nabla)\mathbf b^{\textit R} \right] \to
\\&  (\mathbf u \cdot \nabla)\mathbf b,{\mathcal S}_{\textit R} \left[ (\mathbf b^{\textit R} \cdot \nabla)\mathbf u^{\textit R} \right] \to (\mathbf b \cdot \nabla)\mathbf u,\ \textrm{ strongly\ in} \ L^{\infty}\left([0, T^{\ast}]; \H^{s'-1}(\mathbb R^n)\right), 
\\&\textrm{and} \ 
{\mathcal S}_{\textit R} \left[ (\mathbf u^{\textit R} \cdot \nabla)\mathbf \theta^{\textit R} \right] \to (\mathbf u \cdot \nabla)\mathbf \theta,\ \textrm{strongly\ in} \  L^{\infty}\left([0, T^{\ast}]; H^{s'-1}(\mathbb R^n)\right).
\end{align*}
\end{proposition}

\begin{proof}
We prove the result for one of the non-linear terms, namely 

\noindent ${\mathcal S}_{\textit R} \left[ (\mathbf b^{\textit R}\cdot \nabla)\mathbf b^{\textit R} \right]$, when $s' > n/2 + 1$.  One can follow the similar steps to show the convergence of the other non-linear terms in the respective spaces as claimed. 

Using the properties of Fourier truncation operator and Remark \ref{prodhs} we have,
\begin{align*}
 &\sup_{t \in [0, T^{\ast}]} \|{\mathcal S}_{\textit R} \left[ (\mathbf b^{\textit R} \cdot \nabla)\mathbf b^{\textit R} \right] - (\mathbf b \cdot \nabla)\mathbf b\|_{\H^{s'-1}}
\\ &\leq \sup_{t \in [0,T^{\ast}]}\left( \|{\mathcal S}_{\textit R} \left[ (\mathbf b^{\textit R} - \mathbf b) \cdot \nabla)\mathbf b^{\textit R} \right] \|_{\H^{s'-1}} + \|{\mathcal S}_{\textit R} \left[ (\mathbf b \cdot \nabla)(\mathbf b^{\textit R} - \mathbf b) \right]\|_{\H^{s'-1}} \right)
\\ &\leq \sup_{t \in [0, T^{\ast}]}\left( C \|\left[ (\mathbf b^{\textit R} - \mathbf b) \cdot \nabla)\mathbf b^{\textit R} \right] \|_{\H^{s'-1}} + C \|\left[ (\mathbf b \cdot \nabla)(\mathbf b^{\textit R} - \mathbf b) \right]\|_{\H^{s'-1}} \right)
\\ &\leq \sup_{t \in [0, T^{\ast}]}\left( C \|\mathbf b^{\textit R} - \mathbf b \|_{\H^{s'}} \| \mathbf b^{\textit R}\|_{\H^{s'}} + C \| \mathbf b \|_{\H^{s'}} \| \mathbf b^{\textit R} - \mathbf b\|_{\H^{s'}}\right)
\end{align*}
Clearly from \eqref{urcu}, Proposition \ref{fin} and Proposition \ref{Ca}, the right hand side tends to 0 as $R \to \infty$. 
\end{proof}

Next we will show the convergence of time derivatives.
\begin{proposition}\label{delur}
For any $s' > n/2 +1$, $\frac{\partial \mathbf u^{\textit R}}{\partial t} \to \frac{\partial \mathbf u}{\partial t}$ and $\frac{\partial \mathbf b^{\textit R}}{\partial t} \to \frac{\partial \mathbf b}{\partial t}$ strongly in the space $L^{\infty}\left([0,T^{\ast}]; \H^{s'-1}(\mathbb R^n)\right)$ and $ \frac{\partial \mathbf \theta^{\textit R}}{\partial t}$ converges strongly to $\frac{\partial \mathbf \theta}{\partial t}$ in
 
\noindent $L^{\infty}\left([0, T^{\ast}]; H^{s'-1}(\mathbb R^n)\right)$ as $R \to \infty$.
\end{proposition}

\begin{proof}
Taking $H^{s'-1}$-norm on both sides of  \eqref{tr1}-\eqref{tr3} and using properties of Fourier truncation operator, Remark \ref{prodhs} and Remark \ref{div}, we get for $s' > n/2+1$,
\begin{align*}
\left\| \frac{\partial \mathbf u^{\textit R}}{\partial t} \right\|_{\H^{s'-1}} &\leq \| \theta^{\textit R} e_n\|_{H^{s'-1}} + \|{\mathcal S}_{\textit R} \left[ (\mathbf b^{\textit R} \cdot \nabla)\mathbf b^{\textit R} \right]\|_{\H^{s'-1}} 
\\ &\quad \quad \quad \quad+ \|{\mathcal S}_{\textit R} \left[ (\mathbf u^{\textit R} \cdot \nabla)\mathbf u^{\textit R} \right]\|_{\H^{s'-1}}
\\ &\leq C \left( \| \theta^{\textit R}\|_{H^{s'}} + \|\mathbf b^{\textit R} \|^{2}_{\H^{s'}} + \|\mathbf u^{\textit R} \|^{2}_{\H^{s'}} \right)
\end{align*}
and 
\begin{align*}
\left\| \frac{\partial \mathbf \theta^{\textit R}}{\partial t} \right\|_{\H^{s'-1}} &\leq \| \mathbf u^{\textit R} \cdot e_n\|_{\H^{s'-1}} + \|{\mathcal S}_{\textit R} \left[ (\mathbf u^{\textit R} \cdot \nabla)\mathbf \theta^{\textit R} \right]\|_{\H^{s'-1}} 
\\ &\leq C \left( \|\mathbf u^{\textit R} \|_{\H^{s'}} + \|\mathbf u^{\textit R} \|_{\H^{s'}} \|\mathbf \theta^{\textit R} \|_{H^{s'}} \right)
\end{align*}
and
\begin{align*}
\left\| \frac{\partial \mathbf b^{\textit R}}{\partial t} \right\|_{\H^{s'-1}} &\leq   \|{\mathcal S}_{\textit R} \left[ (\mathbf b^{\textit R} \cdot \nabla)\mathbf u^{\textit R} \right]\|_{\H^{s'-1}} + \|{\mathcal S}_{\textit R} \left[ (\mathbf u^{\textit R} \cdot \nabla)\mathbf b^{\textit R} \right]\|_{\H^{s'-1}} 
\\ &\leq C \|\mathbf b^{\textit R} \|_{\H^{s'}} \|\mathbf u^{\textit R} \|_{\H^{s'}} 
\end{align*}

After adding
\begin{align}\label{dutb}
&\left\| \frac{\partial \mathbf u^{\textit R}}{\partial t} \right\|_{\H^{s'-1}} + \left\| \frac{\partial \mathbf \theta^{\textit R}}{\partial t} \right\|_{\H^{s'-1}} + \left\| \frac{\partial \mathbf b^{\textit R}}{\partial t} \right\|_{\H^{s'-1}} \nonumber
\\ &\leq C( \|\mathbf u^{\textit R} \|_{\H^{s'}} + \| \theta^{\textit R}\|_{H^{s'}} + \|\mathbf u^{\textit R} \|^{2}_{\H^{s'}} + \|\mathbf b^{\textit R} \|^{2}_{\H^{s'}} + \|\mathbf u^{\textit R} \|_{\H^{s'}} \|\mathbf \theta^{\textit R} \|_{H^{s'}} \nonumber
\\ &\quad\quad\quad\quad\quad\quad\quad\quad\quad\quad\quad\quad\quad+ \|\mathbf b^{\textit R} \|_{\H^{s'}} \|\mathbf u^{\textit R} \|_{\H^{s'}} )
\end{align}

Now taking supremum in both side over $t \in [0, T^{\ast}]$, then using Proposition \ref{fin} and dropping the first two terms of left hand side we obtain
\[\sup_{t \in [0, T^{\ast}]} \left\| \frac{\partial \mathbf b^{\textit R}}{\partial t} \right\|_{\H^{s'-1}} \leq C(T^{\ast}) < \infty.\]
Using Banach-Alaoglu Theorem (see Robinson \cite{Ro}, Yosida \cite{Yo}) we can extract a subsequence  $R_m \to +\infty$ such that
\begin{align}\label{drm}
\frac{\partial \mathbf b^{\textit R_m}}{\partial t} \stackrel{*}{\rightharpoonup} \frac{\partial \mathbf b}{\partial t} \quad in \quad L^{\infty}\left([0, T^{\ast}]; \H^{s'-1}(\mathbb R^n)\right).
\end{align}

Similar argument works for $\frac{\partial \mathbf u^{\textit R}}{\partial t}$ and $\frac{\partial \mathbf \theta^{\textit R}}{\partial t}$ as well.

\noindent Note that $\|\mathbf u^{\textit R}\|_{\H^s}\|\mathbf \theta^{\textit R}\|_{H^s} \to \|\mathbf u\|_{\H^s}\|\mathbf \theta\|_{H^s}$ and $\|\mathbf b^{\textit R} \|_{\H^{s'}} \|\mathbf u^{\textit R} \|_{\H^{s'}} \to$ 

\noindent $\|\mathbf b\|_{\H^{s'}} \|\mathbf u\|_{\H^{s'}}$ holds due to the strong convergences of $(\mathbf u^{\textit R}, \mathbf \theta^{\textit R}, \mathbf b^{\textit R})$ to $(\mathbf u, \mathbf \theta, \mathbf b)$ in 

\noindent $L^{\infty}\left([0,  T^{\ast}]; \H^{s'}(\mathbb R^n)\right) \times L^{\infty}\left([0,  T^{\ast}]; H^{s'}(\mathbb R^n)\right) \times L^{\infty}\left([0,  T^{\ast}]; \H^{s'}(\mathbb R^n)\right)$. Hence all the terms on the right hand side of \eqref{dutb} converge strongly (from Proposition \ref{Ca}), 
we observe that the convergence of the time derivatives are strong.
\end{proof}

\begin{proposition}
For $s > n/2 +1,$ $(\mathbf u, \mathbf \theta, \mathbf b)$ lie in the space $L^{\infty}\left([0, T^{\ast}]; \H^{s}(\mathbb R^n)\right) \times L^{\infty}\left([0, T^{\ast}]; H^{s}(\mathbb R^n)\right) \times L^{\infty}\left([0, T^{\ast}]; \H^{s}(\mathbb R^n)\right).$
\end{proposition}

\begin{proof}
By Banach-Alaoglu Theorem, the uniform bounds in Proposition $\ref{fin}$ guarantee the existence of a subsequence such that
\begin{align}
&\mathbf u^{\textit R_m} \stackrel{*}{\rightharpoonup} \mathbf {u}\quad in \quad L^{\infty}\left([0, T^{\ast}]; \H^{s}(\mathbb R^n)\right)
\\ &\mathbf \theta^{\textit R_m} \stackrel{*}{\rightharpoonup} \mathbf \theta\quad in \quad L^{\infty}\left([0, T^{\ast}]; H^{s}(\mathbb R^n)\right)
\end{align}
and
\begin{align}
\mathbf b^{\textit R_m} \stackrel{*}{\rightharpoonup} \mathbf b \quad in \quad L^{\infty}\left([0, T^{\ast}]; \H^{s}(\mathbb R^n)\right),
\end{align}

which guarantees that the limit satisfies
\begin{align}\label{ul2}
 \mathbf u \in  L^{\infty}\left([0,T^{\ast}]; \H^{s}(\mathbb R^n)\right),\quad \mathbf \theta \in L^{\infty}\left([0, T^{\ast}]; H^{s}(\mathbb R^n)\right)
\end{align}
and
\begin{align}
\mathbf b \in L^{\infty}\left([0, T^{\ast}]; \H^{s}(\mathbb R^n)\right)
\end{align}

\end{proof}

\begin{proposition}
Let $(\mathbf u_{0}, \mathbf \theta_{0}, \mathbf b_{0}) \in  \H^{s}(\mathbb R^n) \times H^{s}(\mathbb R^n) \times \H^{s}(\mathbb R^n)$ for $s > n/2 +1.$ Let the solutions $(\u, \theta, \b)$ of the  ideal magnetic B\'{e}nard problem \eqref{mb1}-\eqref{mb4} have the regularity
\[ \mathbf u \in  L^{\infty}\left([0, T^{\ast}]; \H^{s}(\mathbb R^n)\right), \mathbf \theta \in L^{\infty}\left([0, T^{\ast}]; H^{s}(\mathbb R^n)\right), \mathbf b \in  L^{\infty}\left([0, T^{\ast}]; \H^{s}(\mathbb R^n)\right). \]
Then the solutions $(\mathbf u, \mathbf \theta, \mathbf b)$ are unique in $[0, T^{\ast}].$  
\end{proposition}

\begin{proof}
The proof of the uniqueness is very similar to the proof of  Proposition \ref{Ca}. 
Let $(\mathbf u^{\textit R}, \mathbf \theta^{\textit R},  \mathbf b^{\textit R})$ and $(\mathbf u^{\textit R'}, \mathbf \theta^{\textit R'},  \mathbf b^{\textit R'})$ be two solutions of the truncated ideal magnetic B\'{e}nard problem \eqref{tr1}-\eqref{tr3} for $R' > R$. Then from \eqref{ytr}, we have,
\[\sup_{t \in [0, T^{\ast}]}\left(\| \mathbf u^{\textit R}-\mathbf u^{\textit R'}\|_{\L^2} + \| \mathbf \theta^{\textit R} - \mathbf \theta^{\textit R'}\|_{L^2} + \| \mathbf b^{\textit R}-\mathbf b^{\textit R'}\|_{\L^2}\right) \leq \frac{C}{R^\epsilon}.\]
Now letting $R \to R'$ then letting $R \to \infty$ we observe,
\[ \mathbf u^{\textit R} \to \mathbf u^{\textit R'}, \quad \mathbf \theta^{\textit R} \to \mathbf \theta^{\textit R'} \quad and\quad \mathbf b^{\textit R} \to \mathbf b^{\textit R'} .\]
Thus we have the uniqueness of the limits $(\u, \theta, \b)$.
\end{proof}

We finally prove that the solutions $(\u, \theta, \b)$ are continuous in time.
\begin{theorem}\label{mt1}
Let $s > \frac{n}{2} +1,$ $\mathbf u_{0} \in \H^s(\mathbb R^{n})$, $\mathbf \theta_{0} \in  H^s(\mathbb R^{n})$ and $\mathbf b_{0}  \in \H^s(\mathbb R^{n})$. Then there exists a unique strong solution $(\mathbf u, \mathbf \theta, \mathbf b) \in C([0, T^{\ast}]; \H^{s}(\mathbb R^{n})) \times C([0, T^{\ast}]; H^{s}(\mathbb R^{n})) $

\noindent $\times C([0, T^{\ast}]; \H^{s}(\mathbb R^{n}))$ to the system \eqref{mb1}-\eqref{mb4}.
\end{theorem}

\begin{proof}

We shall prove $\mathbf u \in  C \left([0, T^{\ast}]; \H^{s}(\mathbb R^n)\right)$.  Proofs for $\mathbf \theta $ and $\mathbf b$ will follow in the similar manner. 
Let us first recall that for $s \in \mathbb R$, $1 \leq p, q < \infty,$ the inhomogeneous Besov space $B^{s}_{p, q}$ is defined as the space of all tempered distributions $f \in S'(\mathbb R^{n})$ such that
\[B^{s}_{p, q} = \left\{ f \in S'(\mathbb R^{n}) : \|f\|_{B^{s}_{p, q}} < \infty \right\},\]
where
\[ \|f\|_{B^{s}_{p, q}} = \left( \sum_{j \geq -1} 2^{jqs} \|\Delta_{j} f\|^{q}_{L^{p}}\right)^{\frac{1}{q}},\]
where $\Delta_{j}$ are the inhomogeneous Littlewood-Paley operators. We note that $\|f\|_{B^{s}_{2, 2}} \approx  \|f\|_{H^s}.$ For more details see Chapter 3 of \cite{LR}.

We consider $t_1, t_2 \in [0,  T^{\ast}]$ such that $0 \leq t_1 < t_2 \leq  T^{\ast}$. Then,
\[ \|\mathbf u(t_2) - \mathbf u(t_1)\|_{\H^s} \approx \|\mathbf u(t_2) - \mathbf u(t_1)\|_{B^{s}_{2, 2}} = \left\{\sum_{j \in \mathbb Z} \left( 2^{js}\left\| \Delta_j \mathbf u(t_2) - \Delta_j \mathbf u(t_1)\right\|_{\L^2}\right)^2\right\}^{\frac{1}{2}}.\]

Let $\epsilon > 0$ be arbitrarily small. As $\mathbf u \in  L^{\infty} \left([0,T^{\ast}]; H^{s}(\mathbb R^n)\right)$, there exists an integer $N > 0$ such that
\begin{align}\label{tle}
\left\{\sum_{j \geq N} \left(2^{js}\left\| \Delta_j \mathbf u(t_2) - \Delta_j \mathbf u(t_1)\right\|_{\L^2}\right)^2\right\}^{1/2} < \frac{\epsilon}{2}.
\end{align}

But we have
\begin{align*}
&\left\{\sum_{j \in \mathbb Z} \left( 2^{js}\left\| \Delta_j \mathbf u(t_2) - \Delta_j \mathbf u(t_1)\right\|_{\L^2}\right)^2\right\}^{1/2}
\\ &= \left\{ \left(\sum_{j < N} + \sum_{j \geq N}\right) \left( 2^{js}\left\| \Delta_j \mathbf u(t_2) - \Delta_j \mathbf u(t_1)\right\|_{\L^2}\right)^2\right\}^{1/2}.
\end{align*}

Now for $0 \leq t_1 < t_2 \leq T^{\ast}$ we have,
\begin{align*}
\Delta_j \mathbf u(t_2) - \Delta_j \mathbf u(t_1) &= \int_{t_1}^{t_2} \frac{\partial}{\partial \tau} \Delta_j \mathbf u(\tau) \, d\tau 
\\ &= \int_{t_1}^{t_2} \Delta_j \mathcal{P} \left[ (\mathbf b \cdot \nabla) \mathbf b + \theta e_n - (\mathbf u \cdot \nabla) \mathbf u\right](\tau) \, d\tau.
\end{align*}

So we get,
\begin{align}\label{lihs}
\sum_{j < N} &2^{2js}\left.\| \Delta_j \mathbf u(t_2) - \Delta_j \mathbf u(t_1)\|^2_{\L^2}\right. \nonumber
\\ &= \sum_{j < N} 2^{2js}\left.\left\| \int_{t_1}^{t_2} \Delta_j\mathcal{P} \left[ (\mathbf b \cdot \nabla) \mathbf b + \theta e_n - (\mathbf u \cdot \nabla) \mathbf u\right](\tau) \, d\tau \right\|^2_{L^2}\right.\nonumber
\\ &\leq \sum_{j < N} 2^{2js}\left( \int_{t_1}^{t_2} \left[ \| \Delta_j (\mathbf b \cdot \nabla \mathbf b)\|_{\L^2} + \| \Delta_j \theta\|_{L^2} + \| \Delta_j (\mathbf u \cdot \nabla \mathbf u)\|_{\L^2}\right]  \, d\tau \right)^2\nonumber
\\ &= \sum_{j < N} 2^{2j} \Big( \int_{t_1}^{t_2} 2^{j(s-1)}\Big[ \| \Delta_j (\mathbf b \cdot \nabla \mathbf b)\|_{\L^2} + \| \Delta_j \theta\|_{L^2} \nonumber
\\ &\quad \quad \quad\quad+ \| \Delta_j (\mathbf u \cdot \nabla \mathbf u)\|_{\L^2}\Big] \, d\tau \Big)^2\nonumber
\\ &\leq  \sum_{j < N} 2^{2j}\left. \int_{t_1}^{t_2}\left(\| (\mathbf b \cdot \nabla) \mathbf b\|^2_{\H^{s-1}} + \| \theta\|^2_{H^{s-1}} + \|  (\mathbf u \cdot \nabla) \mathbf u\|^2_{\H^{s-1}}\right) \, d\tau \right.
\end{align}

 Clearly the individual terms of right hand side of \eqref{lihs} is further less than their $L^{\infty}\left([0, T^{\ast}]; H^{s-1}\right)-$norm.

As $(\mathbf u, \mathbf \theta, \mathbf b) \in L^{\infty}\left([0, T^{\ast}]; \H^{s}\right) \times L^{\infty}\left([0, T^{\ast}]; H^{s}\right) \times L^{\infty}\left([0, T^{\ast}]; \H^{s}\right)$ and from Remark \ref{prodhs} and Remark \ref{div}  we obtain,
\begin{align*}
\| (\mathbf u \cdot \nabla) \mathbf u\|^2_{L^{\infty}\left([0, T^{\ast}]; \H^{s-1}\right)} &= \left(\sup_{t \in [0, T^{\ast}]}\| (\mathbf u \cdot \nabla) \mathbf u\|_{\H^{s-1}}\right)^2
\\ &\leq \left(\sup_{t \in [0, T^{\ast}]}\| \mathbf u \|_{\H^s} \cdot \sup_{t \in [0, T^{\ast}]}\| \mathbf u\|_{\H^{s}}\right)^2 < C_1 < \infty.
\end{align*}

Similarly, $\| (\mathbf u \cdot \nabla) \mathbf u\|^2_{L^{\infty}\left([0, T^{\ast}]; \H^{s-1}\right)} < C_2$ and  from Proposition \ref{fin}, $\| \theta\|^2_{L^{\infty}\left([0, T^{\ast}]; \H^{s-1}\right)} $

\noindent $< C_3.$ Choosing $M = C \cdot \max\{C_1, C_2, C_3\},$  for $\left.\right.|t_2 - t_1| < \frac{\epsilon}{M {2^{2N+1}}}$, we get from \eqref{lihs}
\begin{align}\label{ile}
\sum_{j < N} 2^{2js}&\left.\| \Delta_j \mathbf u(t_2) - \Delta_j \mathbf u(t_1)\|^2_{\L^2}\right. \leq M \sum_{j < N} 2^{2j} \left|t_2 - t_1\right| \leq M 2^{2N} \left|t_2 - t_1\right| < \frac{\epsilon}{2}. \left.\right.\left.\right. 
\end{align}
Finally combining $\eqref{tle}$ and $\eqref{ile}$, we conclude $\mathbf u \in C\left([0, T^{\ast}]; \H^{s}(\mathbb R^n)\right)$.
\end{proof}

\section{Blow-up criterion}\label{BC}

In this section, we will establish the Blow-up criterion of the local-in-time solution obtained in the previous section. We show that the $BMO$ norms of the vorticity and electrical current inhibit the breakdown of smooth solutions, relaxing the condition on the gradient of temperature, under suitable assumption on the regularity of the initial data. 

\begin{theorem}\label{buc}
Let $(\mathbf u_{0}, \mathbf \theta_{0}, \mathbf b_{0})$ $\in \H^s(\mathbb R^n) \times H^{s}(\mathbb R^n) \times \H^s(\mathbb R^n)$, s $>$ $\frac{n}{2}$+1, n = 2, 3. Let $(\mathbf u, \mathbf \theta, \mathbf b) \in C\left([0, T^{\ast}]; \H^{s}(\mathbb R^n)\right) \times C\left([0, T^{\ast}]; H^{s}(\mathbb R^n)\right) \times C\left([0, T^{\ast}]; \H^{s}(\mathbb R^n)\right)$ be a strong solution of the magnetic B\'enard problem \eqref{mb1}-\eqref{mb4}. If $(\mathbf u,\theta, \mathbf b)$ satisfies the condition 
\begin{equation}
\int_{0}^{T^{\ast}} \left( \|\nabla \times \mathbf u(\tau)\|_{BMO} + \|\nabla \mathbf \theta(\tau)\|_{BMO} + \|\nabla \times \mathbf b(\tau)\|_{BMO}\right) \, d\tau < \infty,
\end{equation}
then the solution $(\mathbf u, \mathbf \theta, \mathbf b)$ can be continuously extended to $[0, T]$ for some $T > T^{\ast}.$ 
\end{theorem}

\begin{proof}
Applying $J^s$ to \eqref{mb1}-\eqref{mb3} and then taking $L^2$-inner product with $J^{s} \mathbf u, J^{s} \mathbf \theta$ and $J^{s} \mathbf b$ respectively, we obtain,  for $s  > \frac{n}{2} +1$, 
\begin{align}\label{jip1}
\left( \frac{\partial (J^s \mathbf u)}{\partial t}, J^s \mathbf u \right)_{L^2} &= \left( J^s \left[ (\mathbf b \cdot \nabla)\mathbf b \right], J^s \mathbf u \right)_{L^2} - \left( J^s \left[ (\mathbf u \cdot \nabla)\mathbf u \right], J^s \mathbf u \right)_{L^2}  \nonumber
\\ &\quad\quad- \left( \nabla J^{s} p_{\ast}, J^s \mathbf u \right)_{L^2} + \left( J^s (\theta e_n), J^s \mathbf u \right)_{L^2},
\end{align}
\begin{align}\label{jip2}
\left( \frac{\partial (J^s \mathbf \theta)}{\partial t}, J^s \mathbf \theta \right)_{L^2} = - \left( J^s \left[ (\mathbf u \cdot \nabla)\mathbf \theta \right], J^s \mathbf \theta \right)_{L^2} +  \left( J^{s} u_n, J^s \mathbf \theta \right)_{L^2},
\end{align}
\begin{align}\label{jip3}
\left( \frac{\partial (J^s \mathbf b)}{\partial t}, J^s \mathbf b \right)_{L^2} = \left( J^s \left[ (\mathbf b \cdot \nabla)\mathbf u \right], J^s \mathbf b \right)_{L^2} -  \left( J^s \left[ (\mathbf u \cdot \nabla)\mathbf b \right], J^s \mathbf b \right)_{L^2}.
\end{align}

Using the definition of commutator, \eqref{jip1}-\eqref{jip3} become
\begin{align}\label{jsc1}
\frac{1}{2} \frac{d}{dt} \|J^s \mathbf u\|^{2}_{\L^2} &= \left( [J^s, \mathbf b] \nabla \mathbf b, J^s \mathbf u \right)_{L^2} + \left( (\mathbf b \cdot \nabla)J^s \mathbf b, J^s \mathbf u \right)_{L^2}\nonumber
\\ &\quad\quad-\left( [J^s, \mathbf u] \nabla \mathbf u, J^s \mathbf u \right)_{L^2} - \left( (\mathbf u \cdot \nabla)J^s \mathbf u, J^s \mathbf u \right)_{L^2} \nonumber
\\ &\quad\quad- \left(J^{s} p_{\ast}, J^s \nabla \cdot \mathbf u \right)_{L^2} + \left( J^s (\theta e_n), J^s \mathbf u \right)_{L^2},
\end{align}
\begin{align}\label{jsc2}
\frac{1}{2} \frac{d}{dt} \|J^s \mathbf \theta\|^{2}_{L^2} = - \left( [J^s, \mathbf u] \nabla \mathbf \theta, J^s \mathbf \theta \right)_{L^2} - \left( (\mathbf u \cdot \nabla)J^s \mathbf \theta, J^s \mathbf \theta \right)_{L^2} + \left( J^{s} u_n, J^s \mathbf \theta \right)_{L^2},
\end{align}
\begin{align}\label{jsc3}
\frac{1}{2} \frac{d}{dt} \|J^s \mathbf b\|^{2}_{\L^2} &= \left( [J^s, \mathbf b] \nabla \mathbf u, J^s \mathbf b \right)_{L^2} + \left( (\mathbf b \cdot \nabla)J^s \mathbf u, J^s \mathbf b \right)_{L^2} \nonumber
\\ &\quad\quad- \left( [J^s, \mathbf u] \nabla \mathbf b, J^s \mathbf b \right)_{L^2} - \left( (\mathbf u \cdot \nabla)J^s \mathbf b, J^s \mathbf b \right)_{L^2}.
\end{align}

Now adding \eqref{jsc1}, \eqref{jsc2} and \eqref{jsc3}, then applying integration by parts, divergence free condition on $\mathbf u$ and $\mathbf b$ and the fact $\left( (\mathbf b \cdot \nabla)J^s \mathbf b, J^s \mathbf u \right)_{L^2} = -\left( (\mathbf b \cdot \nabla)J^s \mathbf u, J^s \mathbf b \right)_{L^2},$ we obtain
\begin{align}\label{eqjs}
\frac{1}{2} \frac{d}{dt} \left( \|\mathbf u\|^{2}_{\H^s} + \|\theta\|^{2}_{H^s} + \|\mathbf b\|^{2}_{\H^s}\right) &= \left( [J^s, \mathbf b] \nabla \mathbf b, J^s \mathbf u \right)_{L^2} -\left( [J^s, \mathbf u] \nabla \mathbf u, J^s \mathbf u \right)_{L^2}\nonumber
\\ &\quad+ \left( J^s (\theta e_n), J^s \mathbf u \right)_{L^2} - \left( [J^s, \mathbf u] \nabla \mathbf \theta, J^s \mathbf \theta \right)_{L^2}\nonumber
\\ &\quad+ \left( J^{s} u_n, J^s \mathbf \theta \right)_{L^2} + \left( [J^s, \mathbf b] \nabla \mathbf u, J^s \mathbf b \right)_{L^2}\nonumber
\\ &\quad- \left( [J^s, \mathbf u] \nabla \mathbf b, J^s \mathbf b \right)_{L^2}.
\end{align}

We will estimate each term on the right hand side of \eqref{eqjs} separately. Using Lemma \ref{kpe}, Remark \ref{gradhs}, Young's inequality and finally rearranging, we obtain,
\begin{align*}
&|\left( [J^s, \mathbf b] \nabla \mathbf b, J^s \mathbf u \right)_{L^2}| \leq \|[J^s, \mathbf b] \nabla \mathbf b\|_{\L^2} \|J^s \mathbf u\|_{\L^2}
\\ &\quad\quad\leq C \left( \| \nabla \mathbf b\|_{L^{\infty}} \| J^{s-1} \nabla \mathbf b \|_{\L^2} + \|J^{s} \mathbf b\|_{\L^2}\| \nabla \mathbf b\|_{L^{\infty}}\right) \|J^{s} \mathbf u\|_{\L^2}
\\ &\quad\quad\leq C \left( \| \nabla \mathbf b\|_{L^{\infty}} \| \nabla \mathbf b \|_{\H^{s-1}} + \|\mathbf b\|_{\H^s} \| \nabla \mathbf b\|_{L^{\infty}}\right) \|\mathbf u\|_{\H^s}
\\ &\quad\quad\leq C \left( \| \nabla \mathbf b\|_{L^{\infty}} \|\mathbf b \|_{\H^{s}} \|\mathbf u\|_{\H^s} + \|\mathbf b\|_{\H^s} \|\mathbf u\|_{\H^s} \| \nabla \mathbf b\|_{L^{\infty}}\right)
\\ &\quad\quad\leq C \| \nabla \mathbf b\|_{L^{\infty}} (\|\mathbf u\|_{\H^s} \|\mathbf b\|_{\H^s})
\\ &\quad\quad\leq C (\| \nabla \mathbf u\|_{L^{\infty}} + \| \nabla \mathbf \theta\|_{L^{\infty}} + \| \nabla \mathbf b\|_{L^{\infty}})(\|\mathbf u\|^{2}_{\H^s} + \|\mathbf \theta\|^{2}_{H^s} + \|\mathbf b\|^{2}_{\H^s}).
\end{align*}

Similarly the other commutator terms on the right hand side of \eqref{eqjs} can be estimated. 
\begin{align*}
&|([J^s, \mathbf u] \nabla \mathbf u, J^s \mathbf u)_{L^2}|
\\ &\leq C (\| \nabla \mathbf u\|_{L^{\infty}} + \| \nabla \mathbf \theta\|_{L^{\infty}} + \| \nabla \mathbf b\|_{L^{\infty}})(\|\mathbf u\|^{2}_{\H^s} + \|\mathbf \theta\|^{2}_{H^s} + \|\mathbf b\|^{2}_{\H^s}). 
\end{align*}
\begin{align*}
|( &[J^s, \mathbf u] \nabla \mathbf \theta, J^s \mathbf \theta)_{L^2}|
\\ &\leq C (\| \nabla \mathbf u\|_{L^{\infty}} + \| \nabla \mathbf \theta\|_{L^{\infty}} + \| \nabla \mathbf b\|_{L^{\infty}})(\|\mathbf u\|^{2}_{\H^s} + \|\mathbf \theta\|^{2}_{H^s} + \|\mathbf b\|^{2}_{\H^s}). 
\end{align*}
\begin{align*}
|( &[J^s, \mathbf b] \nabla \mathbf u, J^s \mathbf b)_{L^2}|
\\ &\leq C (\| \nabla \mathbf u\|_{L^{\infty}} + \| \nabla \mathbf \theta\|_{L^{\infty}} + \| \nabla \mathbf b\|_{L^{\infty}})(\|\mathbf u\|^{2}_{\H^s} + \|\mathbf \theta\|^{2}_{H^s} + \|\mathbf b\|^{2}_{\H^s}). 
\end{align*}
\begin{align*}
|( &[J^s, \mathbf u] \nabla \mathbf b, J^s \mathbf b)_{L^2}|
\\ &\leq C (\| \nabla \mathbf u\|_{L^{\infty}} + \| \nabla \mathbf \theta\|_{L^{\infty}} + \| \nabla \mathbf b\|_{L^{\infty}})(\|\mathbf u\|^{2}_{\H^s} + \|\mathbf \theta\|^{2}_{H^s} + \|\mathbf b\|^{2}_{\H^s}). 
\end{align*}

Now
\begin{align*}
|\left( J^s (\theta e_n), J^s \mathbf u \right)_{L^2}|  &\leq \| J^s (\theta e_n)\|_{\L^{2}} \|J^s \mathbf u\|_{\L^{2}} \leq  \|\theta\|_{H^s} \|\mathbf u\|_{\H^s} 
\\ &\leq C (\|\mathbf u\|^{2}_{\H^s} + \|\mathbf \theta\|^{2}_{H^s} + \|\mathbf b\|^{2}_{\H^s}) .
\end{align*}
and
\begin{align*}
|\left( J^{s} \mathbf u_n, J^s \mathbf \theta \right)_{L^2}| \leq C (\|\mathbf u\|^{2}_{\H^s} + \|\mathbf \theta\|^{2}_{H^s} + \|\mathbf b\|^{2}_{\H^s}).
\end{align*}

Combining all the estimates above,  from \eqref{eqjs} after taking $X(t) = \|\mathbf u(t)\|^{2}_{\H^s} + \|\mathbf \theta(t)\|^{2}_{H^s} + \|\mathbf b(t)\|^{2}_{\H^s}$ for $t \in [0, T^{\ast}],$ we obtain
\begin{align*}
\frac{d}{dt}X(t) \leq C (\| \nabla \mathbf u\|_{L^{\infty}} + \| \nabla \mathbf \theta\|_{L^{\infty}} + \| \nabla \mathbf b\|_{L^{\infty}} +2) X(t).
\end{align*}
Standard Gronwall's inequality yields
\[ X(t) \leq X(0) \left.\right. \exp \left( C \int_{0}^{t} (\|\nabla \mathbf u(\tau)\|_{L^{\infty}}+ \|\nabla \mathbf \theta(\tau)\|_{L^{\infty}}+ \| \nabla \mathbf b(\tau)\|_{L^{\infty}} + 2) \, d\tau \right).\]
Hence
\begin{align}\label{ee3}
 X(t)\leq X(0) \exp \left( C \int_{0}^{t} (\|\nabla \mathbf u(\tau)\|_{L^{\infty}}+ \|\nabla \mathbf \theta(\tau)\|_{L^{\infty}} +\| \nabla \mathbf b(\tau)\|_{L^{\infty}} + 2) \, d\tau \right).
\end{align}
Due to the logarithmic Sobolev inequality in Lemma \ref{kte}, and the fact that singular integral operators of Calderon-Zygmund type are bounded in $BMO$ (i.e. $\|\nabla \mathbf u\|_{BMO} \leq \|\nabla \times \mathbf u\|_{BMO}$), for $s > \frac{n}{2} + 1$ we obtain, 
\begin{align}\label{bmou}
\|\nabla \mathbf u\|_{L^{\infty}} &\leq C\left( 1 + \|\nabla \mathbf u\|_{BMO}\left(1 + \log^{+}\|\nabla \mathbf u\|_{ \H^{s-1}}\right) \right) \nonumber
\\ &\leq C\left( 1 + \|\nabla \times \mathbf u\|_{BMO}\left(1 + \log^{+}\|\mathbf u\|_{\H^{s}}\right) \right) \nonumber
\\ &\leq C\left( 1 + \|\nabla \times \mathbf u\|_{BMO}\left(1 + \frac{1}{2} \log^{+}\|\mathbf u\|^{2}_{\H^{s}}\right) \right) \nonumber
\\ &\leq C\left( 1 + \|\nabla \times \mathbf u\|_{BMO}\left(1 + \frac{1}{2} \log^{+} \left(\|\mathbf u\|^{2}_{\H^{s}} + \|\mathbf \theta\|^{2}_{H^{s}}+ \|\mathbf b\|^{2}_{\H^s}\right) \right) \right) \nonumber
\\ &\leq C\left( 1 + \|\nabla \times \mathbf u\|_{BMO}\left(1 + \log^{+} \left(\|\mathbf u\|^{2}_{\H^{s}} + \|\mathbf \theta\|^{2}_{H^{s}}+ \|\mathbf b\|^{2}_{\H^s}\right) \right) \right).
\end{align}
Similarly we obtain,
\begin{align}\label{bmot}
\|\nabla \mathbf \theta\|_{L^{\infty}} \leq C\left( 1 + \|\nabla \mathbf \theta\|_{BMO}\left(1 + \log^{+} \left(\|\mathbf u\|^{2}_{\H^{s}} + \|\mathbf \theta\|^{2}_{H^{s}}+ \|\mathbf b\|^{2}_{\H^s}\right) \right) \right)
\end{align}
and
\begin{align}\label{bmob}
\|\nabla \mathbf b\|_{L^{\infty}} \leq C\left( 1 + \|\nabla \times \mathbf b\|_{BMO}\left(1 + \log^{+} \left(\|\mathbf u\|^{2}_{\H^{s}} + \|\mathbf \theta\|^{2}_{H^{s}}+ \|\mathbf b\|^{2}_{\H^s}\right) \right) \right)
\end{align}
Now using \eqref{bmou}, \eqref{bmot} and \eqref{bmob} in \eqref{ee3}, we obtain for all $t \in [0, T^{\ast}]$,
\begin{align*}
X(t)&\leq X(0)\exp\Big[C\int_{0}^{t} \Big\{5 + (\|\nabla \times \mathbf u(\tau)\|_{BMO}+ \|\nabla \mathbf \theta(\tau)\|_{BMO} 
\\ & \qquad \qquad + \|\nabla \times \mathbf b(\tau)\|_{BMO}) \times\left(1 + \log^{+} X(\tau) \right)\, \Big\}d\tau\Big].
\end{align*}
Taking ``log" on both sides we get for all $t \in [0, T^{\ast}]$,
\begin{align*}
\log X(t) &\leq \log X(0) + C \int_{0}^{t} \Big\{5 + (\|\nabla \times \mathbf u(\tau)\|_{BMO}+ \|\nabla \mathbf \theta(\tau)\|_{BMO} 
\\&\qquad\qquad + \|\nabla \times \mathbf b(\tau)\|_{BMO}) \times (1 + \log^{+} X(\tau))\Big\} \, d\tau.
\end{align*}
Rearranging the terms we have
\begin{align*}
\log (eX(t)) & \leq \log (eX(0)) + C T^{\ast} + \int_{0}^{t} \Big\{(\|\nabla \times \mathbf u(\tau)\|_{BMO}+ \|\nabla \mathbf \theta(\tau)\|_{BMO}\\
&\quad + \|\nabla \times \mathbf b(\tau)\|_{BMO}) (\log (eX(\tau)))\Big\} \, d\tau.
\end{align*}
Now Gronwall's inequality yields
\begin{align*}
\log (eX(t)) &\leq \Big\{(\log (eX(0)) + CT^{\ast}) \times \exp \Big( C \int_{0}^{t} (\|\nabla \times \mathbf u(\tau)\|_{BMO}
\\&\quad\quad + \|\nabla \mathbf \theta(\tau)\|_{BMO}+ \|\nabla \times \mathbf b(\tau)\|_{BMO}) \, d\tau \Big)\Big\}.
\end{align*}
Taking supremum over all $t \in [0, T^{\ast}]$ we obtain,
\begin{align*}
\sup_{t \in [0, T^{\ast}]} \log X(t) &\leq \sup_{t \in [0, T^{\ast}]} \log (eX(t))\\&\quad \leq (\log (eX(0)) + CT^{\ast}) \times \exp \Big( C \int_{0}^{T^{\ast}} (\|\nabla \times \mathbf u(\tau)\|_{BMO}\\& \qquad+ \|\nabla \mathbf \theta(\tau)\|_{BMO}+ \|\nabla \times \mathbf b(\tau)\|_{BMO}) \, d\tau \Big).
\end{align*}
So finally we acquire,
\begin{align*}
\sup_{t \in [0, T^{\ast}]} X(t) &\leq e^{(1+CT^{\ast})} X(0) \times \exp\Big\{ \exp\Big( C \int_{0}^{T^{\ast}} (\|\nabla \times \mathbf u(\tau)\|_{BMO}
\\ &\quad\quad + \|\nabla \mathbf \theta(\tau)\|_{BMO}+ \|\nabla \times \mathbf b(\tau)\|_{BMO}) \, d\tau \Big)\Big\}.
\end{align*}
This concludes that if 
\[\int_{0}^{T^{\ast}} (\|\nabla \times \mathbf u(\tau)\|_{BMO}+ \|\nabla \mathbf \theta(\tau)\|_{BMO}+ \|\nabla \times \mathbf b(\tau)\|_{BMO}) \, d\tau < \infty,\]
then by continuation of local solutions, we can extend the solution to $[0, T]$ for some $T > T^{\ast}$.

\end{proof}

We now show that the assumption we made in Theorem \ref{buc} for $\nabla\mathbf\theta$ can be relaxed completely. In other words,  provided $ \mathbf \theta_0 \in H^s(\mathbb R^{n}) \cap W^{1, p}(\mathbb R^{n}), p\geq 2$, the bound on curl of $\mathbf u$ and curl of $\mathbf b$  are enough to extend the solution continuously to some time $T > T^{\ast}$.\\
\noindent
Before proving the above result let us note the following vector identity.
\begin{remark}\label{vi}
\begin{align*}
\nabla (\mathbf u \cdot \nabla \mathbf \theta) &= (\mathbf u \cdot \nabla) \nabla \mathbf \theta + (\nabla \mathbf \theta \cdot  \nabla) \mathbf u + \mathbf u \times (\nabla \times \nabla \mathbf \theta) + \nabla \mathbf \theta \times (\nabla \times \mathbf u)
\\ &= (\mathbf u \cdot \nabla) \nabla \mathbf \theta + (\nabla \mathbf \theta \cdot  \nabla) \mathbf u + \nabla \mathbf \theta \times (\nabla \times \mathbf u)
\\ &=  (\mathbf u \cdot \nabla) \nabla \mathbf \theta + (\nabla \mathbf u)^{t} \cdot \nabla \mathbf \theta 
\end{align*}
where we have used the facts that curl of the gradient of a scalar function is zero (i.e., $\mathbf u \times (\nabla \times \nabla \mathbf \theta) = 0$)  and $(\nabla \mathbf u)^{t} \cdot \nabla \mathbf \theta = (\nabla \mathbf \theta \cdot  \nabla) \mathbf u + \nabla \mathbf \theta \times (\nabla \times \mathbf u).$.
\end{remark}

\begin{theorem}\label{buc1}
Let $s > \frac{n}{2} + 1$, $\mathbf u_{0} \in \H^{s}(\mathbb R^{n})$, $\mathbf b_{0} \in \H^s(\mathbb R^n)$ and $\mathbf \theta_0 \in H^s(\mathbb R^{n}) \cap W^{1, p}(\mathbb R^{n}),$ for $ 2 \leq p \leq \infty,$ n=2, 3. Let $(\mathbf u, \mathbf \theta, \mathbf b) \in C\left([0, T^{\ast}]; \H^{s}(\mathbb R^n)\right) \times C\left([0, T^{\ast}]; H^{s}(\mathbb R^n)\right) $

\noindent $\times C\left([0, T^{\ast}]; H^{s}(\mathbb R^n)\right)$ be a strong solution of the ideal magnetic B\'enard problem \eqref{mb1}-\eqref{mb4}. Then 
\[\int_{0}^{T^{\ast}} \|\nabla \times \mathbf u(\tau)\|_{BMO} + \|\nabla \times \mathbf b(\tau)\|_{BMO} \,d\tau < \infty\]
guarantees that the solution can be extended continuously to $[0, T]$ for some $T > T^{\ast}.$
\end{theorem}

\begin{proof}
We consider the equation \eqref{mb2} as follows,
\[\frac{\partial \mathbf \theta}{\partial t} + (\mathbf u \cdot \nabla)\mathbf \theta = \mathbf {u}_n,\]
and apply the gradient operator $\nabla = (\partial_{x_1}, \dots, \partial_{x_n})$ on both sides and take $L^2$-inner product with $\nabla \mathbf \theta |\nabla \mathbf \theta|^{p-2}$ to obtain,
\[\left(\frac{\partial}{\partial t} (\nabla \mathbf \theta), \nabla \mathbf \theta |\nabla \mathbf \theta|^{p-2}\right) + \left(\nabla (\mathbf u \cdot \nabla \mathbf \theta), \nabla \mathbf \theta |\nabla \mathbf \theta|^{p-2}\right) =\left(\nabla \mathbf u_n, \nabla \mathbf \theta |\nabla \mathbf \theta|^{p-2}\right).\]
Using the vector identity in Remark \ref{vi} we obtain,
\begin{align}\label{tttt}
\left(\frac{\partial}{\partial t} (\nabla \mathbf \theta), \nabla \mathbf \theta |\nabla \mathbf \theta|^{p-2}\right) + \left( (\nabla \mathbf u)^{t} \cdot \nabla \mathbf \theta, \nabla \mathbf \theta |\nabla \mathbf \theta|^{p-2}\right) &+ \left( (\mathbf u \cdot \nabla) \nabla \mathbf \theta, \nabla \mathbf \theta |\nabla \mathbf \theta|^{p-2}\right) \nonumber
\\ &=\left(\nabla \mathbf u_n, \nabla \mathbf \theta |\nabla \mathbf \theta|^{p-2}\right).
\end{align}
We will calculate each term separately. The first term of left hand side of \eqref{tttt} gives,
\[\left(\frac{\partial}{\partial t} (\nabla \mathbf \theta), \nabla \mathbf \theta |\nabla \mathbf \theta|^{p-2}\right) = \frac{1}{p} \int_{\mathbb R^{n}} \frac{\partial}{\partial t}  |\nabla \mathbf \theta|^{p} \,dx = \frac{1}{p} \frac{d}{dt} \| \nabla \mathbf \theta\|_{L^p}^{p}\]
and
\begin{align*}
\left( (\nabla \mathbf u)^{t} \cdot \nabla \mathbf \theta, \nabla \mathbf \theta |\nabla \mathbf \theta|^{p-2}\right) &= \int_{\mathbb R^{n}} (\nabla \mathbf u)^{t} \cdot \nabla \mathbf \theta \cdot \nabla \mathbf \theta |\nabla \mathbf \theta|^{p-2} \,dx 
\\ &\leq \int_{\mathbb R^{n}} (\nabla \mathbf u)^{t} \cdot |\nabla \mathbf \theta|^{p}
\leq \|\nabla \mathbf u\|_{L^{\infty}} \|\nabla \mathbf \theta\|_{L^p}^{p}.
\end{align*}
By applying integration by parts and the divergence free condition of $\mathbf u$, we have from the third term of \eqref{tttt},
\begin{align*}
\left( (\mathbf u \cdot \nabla) \nabla \mathbf \theta, \nabla \mathbf \theta |\nabla \mathbf \theta|^{p-2}\right) &= \int_{\mathbb R^{n}} (\mathbf u \cdot \nabla) \nabla \mathbf \theta \cdot \nabla \mathbf \theta |\nabla \mathbf \theta|^{p-2} \,dx
\\ &= \frac{1}{p}  \int_{\mathbb R^{n}} \mathbf u \cdot \nabla |\nabla \mathbf \theta|^{p} \,dx
= -\frac{1}{p} \int_{\mathbb R^{n}} (\nabla \cdot \mathbf u) \cdot |\nabla \mathbf \theta|^{p} \,dx = 0.
\end{align*}
Now 
\begin{align*}
\left|\left(\nabla \mathbf u_n, \nabla \mathbf \theta |\nabla \mathbf \theta|^{p-2}\right)\right| &= \left|\int_{\mathbb R^{n}} \nabla \mathbf u_n \cdot \nabla \mathbf \theta |\nabla \mathbf \theta|^{p-2} \, dx \right| \leq \int_{\mathbb R^{n}} |\nabla \mathbf u_n | |\nabla \mathbf \theta|^{p-1} \, dx
\\ &\leq \left( \int_{\mathbb R^{n}}  |\nabla \mathbf u_n |^{p}\right)^{\frac{1}{p}} \left( \int_{\mathbb R^{n}} ( |\nabla \mathbf \theta|^{p-1})^{\frac{p}{p-1}}\right)^{\frac{p-1}{p}}
\\ &\leq \| \nabla \mathbf u_n \|_{L^p} \|\nabla \mathbf \theta\|_{L^p}^{p-1} \leq \frac{1}{p} \left(\| \nabla \mathbf u_n \|_{L^p}^{p} + (p-1) \|\nabla \mathbf \theta\|_{L^p}^{p}\right)
\end{align*}
Due to the Gagliardo-Nirenberg interpolation inequality (see Lemma \ref{gni}), while $j=1, m=3, r=2, q=2$, we have for 
$n=2$, 
\begin{align} \label{n2}
\| \nabla \mathbf u_2 \|_{L^p} \leq C \|\mathbf u_2 \|_{L^2}^{\frac{2+p}{3p}}  \|\mathbf u_2 \|_{H^3}^{\frac{2p-2}{3p}}, \quad \quad p \geq 2,
\end{align}
and for $n=3$,
\begin{align} 
\| \nabla \mathbf u_3 \|_{L^p} \leq C \|\mathbf u_3 \|_{L^2}^{\frac{6+p}{6p}}  \|\mathbf u_3 \|_{H^3}^{\frac{5p-6}{6p}}, \quad \quad  p \geq 2.
\end{align}
From the term-wise estimates of \eqref{tttt}, we obtain when $n=3$,
\[\frac{d}{dt} \| \nabla \mathbf \theta\|_{L^p}^{p} \leq  p \|\nabla \mathbf u\|_{L^{\infty}} \|\nabla \mathbf \theta\|_{L^p}^{p} + C^{p} \|\mathbf u_3 \|_{L^2}^{\frac{6+p}{6}}  \|\mathbf u_3 \|_{H^3}^{\frac{5p-6}{6}} + (p-1) \| \nabla \mathbf \theta\|_{L^p}^{p},\]
which further gives due to Gronwall's inequality,
\begin{align*}
\| \nabla \mathbf \theta\|_{L^p}^{p} &\leq  \left(\| \nabla \mathbf \theta_{0}\|_{L^p}^{p} +  C^{p} \int_{0}^{t} \|\mathbf u_3(\tau) \|_{L^2}^{\frac{6+p}{6}}  \|\mathbf u_3(\tau) \|_{H^3}^{\frac{5p-6}{6}} \, d\tau \right) 
\\ &\ \ \ \ \ \ \ \ \ \ \  \times \exp \left(  \int_{0}^{t} (p\|\nabla \mathbf u(\tau)\|_{L^{\infty}}+ p-1) \,d\tau \right)
\\ &\leq \left(\| \nabla \mathbf \theta_{0}\|_{L^p}^{p} +  C^{p} \ T^{\ast} \left(\sup_{t \in [0, T^{\ast}]} \|\mathbf u_3 \|_{L^2}\right)^{\frac{6+p}{6}}  \left(\sup_{t \in [0, T^{\ast}]} \|\mathbf u_3\|_{H^3}\right)^{\frac{5p-6}{6}} \right)
\\ &\ \ \ \ \ \ \ \ \ \ \  \times \exp (pT^{\ast}) \  \exp \left(p \int_{0}^{t} \|\nabla \mathbf u(\tau)\|_{L^{\infty}} \, d\tau\right).
\end{align*}
Therefore we obtain, when $n=3$,
\begin{align*}
\| \nabla \mathbf \theta\|_{L^p} &\leq \left(\| \nabla \mathbf \theta_{0}\|_{L^p}^{p} +  C^{p} \ T^{\ast} \left(\sup_{t \in [0, T^{\ast}]} \|\mathbf u_3 \|_{L^2}\right)^{\frac{6+p}{6}}  \left(\sup_{t \in [0, T^{\ast}]} \|\mathbf u_3\|_{H^3}\right)^{\frac{5p-6}{6}} \right)^{\frac{1}{p}}
\\ &\ \ \ \ \ \ \ \ \ \ \  \times \exp (T^{\ast}) \  \exp \left( \int_{0}^{t} \|\nabla \mathbf u(\tau)\|_{L^{\infty}} \, d\tau\right)
\\ &\leq \left(\| \nabla \mathbf \theta_{0}\|_{L^p} +  C T^{{\ast}^{1/p}} \left(\sup_{t \in [0, T^{\ast}]} \|\mathbf u_3 \|_{L^2}\right)^{\frac{6+p}{6p}}  \left(\sup_{t \in [0, T^{\ast}]} \|\mathbf u_3\|_{H^3}\right)^{\frac{5p-6}{6p}} \right)
\\ &\ \ \ \ \ \ \ \ \ \ \  \times \exp (T^{\ast}) \  \exp \left( \int_{0}^{t} \|\nabla \mathbf u(\tau)\|_{L^{\infty}} \, d\tau\right).
\end{align*}
Since the $L^2$-energy estimate and $H^3$-energy estimate of $\mathbf u$ are finite due to Propositions \ref{lfin} and \ref{fin}, letting $p \to \infty,$ we finally obtain
\begin{align*}
\| \nabla \mathbf \theta\|_{L^\infty} \leq  C(T^{\ast}) \  \| \nabla \mathbf \theta_{0}\|_{L^\infty} \exp \left( \int_{0}^{t} \|\nabla \mathbf u(\tau)\|_{L^{\infty}} \,d\tau \right).
\end{align*}
Similarly, the case $n=2$ (from \eqref{n2}) will also yield the same above estimate.\\
\noindent
Note that, due to Lemma \ref{kte}, and properties of the $BMO$ space, we further have,
\begin{align*}
\| \nabla \mathbf \theta\|_{L^\infty} \leq  \| \nabla \mathbf \theta_{0}\|_{L^\infty} \exp \left( C \int_{0}^{t} \left( 1 + \|\nabla \times \mathbf{u}(\tau)\|_{BMO}\left(1 + log^{+}\|\mathbf{u}(\tau)\|_{\H^{s}}\right) \right) \,d\tau \right).
\end{align*}
 As  $ \mathbf \theta_{0} \in H^s(\mathbb R^{n}) \cap W^{1, p}(\mathbb R^{n}), 2\leq p \leq\infty$ and $\sup_{t \in [0, T^{\ast}]} \|\mathbf u\|_{\H^{s}}$ is bounded for $s>n/2+1$, we have,
\begin{align}\label{tw}
\| \nabla \mathbf \theta\|_{L^\infty} \leq C \exp \left(\int_{0}^{T^{\ast}}  \|\nabla \times \mathbf{u}(\tau)\|_{BMO} \, d\tau\right) 
\end{align}
where $C = C(\|\nabla \mathbf \theta_{0}\|_{L^\infty}, \|\mathbf u\|_{\H^{s}}, T^{\ast})$. \\
\noindent
Due to the assumption $\int_{0}^{T^{\ast}} \|\nabla \times \mathbf u(\tau)\|_{BMO} \,d\tau < \infty$, the estimate in \eqref{tw} is bounded. Hence, $\| \nabla \mathbf \theta\|_{BMO} \leq 2 \| \nabla \mathbf \theta\|_{L^\infty} \leq C < \infty.$
So the bound on $BMO$ norms of vorticity and electrical current are enough to guarantee that the solution can be extended to $[0, T]$ for some $T > T^{\ast}$ provided $ \mathbf \theta_{0} \in H^s(\mathbb R^{n}) \cap W^{1, p}(\mathbb R^{n}).$ 
\end{proof}

\medskip\noindent
{\bf Acknowledgements:} Utpal Manna's work has been supported  by National
Board of Higher Mathematics (NBHM), Govt. of India. Both
the authors would like to thank Indian Institute of Science Education
and Research Thiruvananthapuram for providing stimulating scientific environment
and resources.

\end{document}